 \documentclass{amsart}

\usepackage{amsrefs}
\usepackage{euscript}
\usepackage{amsmath}
\usepackage{amscd}
\usepackage{amssymb}
\usepackage{enumerate}
\usepackage{stmaryrd}
\usepackage{amsfonts}
\usepackage{graphicx}
\usepackage[all]{xy}
\usepackage[margin=1.5in]{geometry}
\numberwithin{equation}{section}

\newcommand{\bR}{{\mathbb R}}
\newcommand{\bZ}{{\mathbb Z}}

\newcommand{\bF}{{\mathbb F}}

\newcommand{\sQ}{\EuScript Q}
\newcommand{\csQ}{\check{\EuScript Q}}
\newcommand{\csigma}{\check{\sigma}}
\newcommand{\ctau}{\check{\tau}}
\newcommand{\crho}{\check{\rho}}

\newcommand{\cY}{\mathcal Y}

\newcommand{\Tree}{\EuScript T}

\newcommand{\TN}{T^{*}N}
\newcommand{\Tn}{T^{*}_{n} N}
\newcommand{\TtN}{T^{*}\tN}
\newcommand{\Ttn}{T^{*}_{\tn} \tN}

\newcommand{\tM}{\tilde{M}}
\newcommand{\tL}{\tilde{L}}
\newcommand{\tN}{\tilde{N}}
\newcommand{\tQ}{\tilde{Q}}
\newcommand{\tn}{\tilde{n}}
\newcommand{\tx}{\tilde{x}}
\newcommand{\ty}{\tilde{y}}
\newcommand{\tz}{\tilde{z}}

\newcommand{\tE}{\tilde{E}}
\newcommand{\tu}{\tilde{u}}

\newcommand{\inte}{int}

\newcommand{\vx}[1][\!]{\vec{x}^{\, #1}}
\newcommand{\vtx}[1][\!]{\vec{\tilde{x}}^{\, #1}}

\newcommand{\Chord}{{\EuScript X}}

\newcommand{\colim}{\operatorname*{colim}}

\newcommand{\id}{\operatorname{id}}

\renewcommand{\mod}{\operatorname{mod}}

\newcommand{\Disc}{{\mathcal R}}

\newcommand{\Discbar}{\overline{\Disc}}

\newcommand{\Ann}{{\mathcal C}}

\newcommand{\OC}{\mathcal {OC}}
\newcommand{\CO}{\mathcal {CO}}
\newcommand{\Wrap}{\EuScript{W}}
\newcommand{\Seidel}{\EuScript{S}}
\newcommand{\tWrap}{\tilde{\EuScript{W}}}
\newcommand{\tSeidel}{\tilde{\Seidel}}

\newcommand{\Hom}{\operatorname{Hom}}
\newcommand{\End}{\operatorname{End}}

\newcommand{\Action}{\mathcal{A}}

\newcommand{\Univ}{\mathcal{U}}
\newcommand{\Bimod}{\operatorname{P}}

\def\co{\colon\thinspace}

\newtheorem{thm}{Theorem}[section]
\newtheorem{cor}[thm]{Corollary}
\newtheorem{lem}[thm]{Lemma}
\newtheorem{prop}[thm]{Proposition}
\newtheorem{defin}[thm]{Definition}
\newtheorem{def-lem}[thm]{Definition-Lemma}

\theoremstyle{remark}

\newtheorem{rem}[thm]{Remark}

\newtheorem{example}[thm]{Example}

\newcommand{\superscript}[1]{\ensuremath{^{\textrm{#1}}} }

\renewcommand{\th}[0]{\superscript{th}}

\newcommand{\comment}[1]{}

\title[Nearby Lagrangians with vanishing Maslov class are homotopy equivalent]{Nearby Lagrangians with vanishing Maslov class are homotopy equivalent}

\author[M.~Abouzaid]{Mohammed Abouzaid} \date{\today}
\thanks{This research was conducted during the period the author served as a Clay Research Fellow. }

\begin{document}

\begin{abstract}
We prove that the inclusion of every closed exact Lagrangian with vanishing Maslov class in a cotangent bundle is a homotopy equivalence.  We start by adapting an idea of Fukaya-Seidel-Smith to prove that such a Lagrangian is equivalent to the zero section in the Fukaya category with integral coefficients.  We then study an extension of the Fukaya category in which Lagrangians equipped with local systems of arbitrary dimension are admitted as objects, and prove that this extension is generated, in the appropriate sense, by local systems over a cotangent fibre.  Whenever the cotangent bundle is simply connected, this generation statement is used to prove that every closed exact Lagrangian of vanishing Maslov index is simply connected.  Finally, we borrow ideas from coarse geometry to develop a Fukaya category associated to the universal cover, allowing us to prove the result in the general case.
\end{abstract}
\maketitle
\setcounter{tocdepth}{1}
\tableofcontents

\section{Introduction}
Arnol'd conjectured that every closed exact Lagrangian in the cotangent bundle of a closed manifold is Hamiltonian isotopic to the zero section.  The main purpose of this paper is to prove the following result:
\begin{thm}\label{thm:simply_connected}
If $N$ is a closed manifold, and $Q \subset \TN$ is a closed exact Lagrangian whose Maslov class vanishes, then the inclusion of $Q$ induces an isomorphism on fundamental groups.
\end{thm}
In \cite{FSS2}, Fukaya, Seidel, and Smith proved, assuming certain conjectural facts about the Fukaya categories of cotangent bundles, that every such inclusion induces an isomorphism on homology when the Maslov index vanishes, and $N$ is orientable.  We shall explain a modified version of their proof in Appendix \ref{sec:FSS}.  
\begin{cor}
Under the assumptions of Theorem \ref{thm:simply_connected}, the inclusion of $Q$ is a homotopy equivalence.
\end{cor}
\begin{proof}
The Whitehead and Hurewicz  theorems imply that a map which induces an isomorphism on fundamental groups and on cohomology is a homotopy equivalence, proving the result whenever $N$ is orientable.  Otherwise, we pass to the orientation cover $\tilde{N}$ of $N$, and let $\tQ$ denote the inverse image of $Q$ in $\TtN$, to which we can apply the argument for oriented bases, and conclude that the inclusion $\tQ \subset   \TtN$ is a homotopy equivalence.  In particular, $ \tQ  $ is connected, so the composition of the homomorphisms
\begin{equation}
  \pi_{1}(Q) \to \pi_{1}(\TN) \to \bZ_{2} 
\end{equation}
is surjective, where the second homomorphism is the one associated to the cover $\TtN $.  By considering the commutative diagram
\begin{equation}
  \xymatrix{ 1 \ar[r] &   \pi_{1}(\tQ)  \ar[r] \ar[d]^{\cong} & \pi_{1}(Q) \ar[r] \ar[d] &  \bZ_{2} \ar[r]  \ar[d]^{=} & 1  \\
1 \ar[r] &   \pi_{1}(\TtN)  \ar[r] \ar[r] & \pi_{1}(\TN) \ar[r] &  \bZ_{2} \ar[r]  & 1}
\end{equation}
we conclude that the inclusion of $Q$ in $\TN$ induces an isomorphism on fundamental groups.   Using the fact that the higher homotopy groups are invariant under passage to a cover, and that they are isomorphic for $\TtN$ and $\tQ$, we conclude from the Whitehead theorem that the inclusion of $Q$ in $\TN$ is a homotopy equivalence.
\end{proof}

Let $M$ be a Liouville manifold, and write $CW^{*}(L)$ for the self-Floer complex of some exact Lagrangian $L$.  As this is an $A_{\infty}$ algebra, it has a Hochschild homology group $HH_{*}(CW^*(L))$, which is the source of a map 
\begin{equation}
H^{*}(\OC) \co  HH_{*-n}( CW^{*}(L) )  \to SH^{*}(M)
\end{equation}
constructed in \cite{generate}, whose target is symplectic cohomology.  We say that $L$ \emph{resolves the diagonal} if the identity lies in the image of this map.  The main result of \cite{generate} is that this condition implies that every exact Lagrangian in $M$ which, away from a compact set, is a cone on a Legendrian manifold of the ideal contact boundary of $M$, can be obtained from $L$  by taking cones and summands; i.e. that  $L$ split-generates the wrapped Fukaya category.

In this paper, we shall enlarge the wrapped Fukaya category to a category we denote $\Seidel(M)$ whose objects are Lagrangians equipped with $\bF_{2}$ local systems of arbitrary dimension (the local systems are not equipped with any topology).  By dimension, we mean the cardinality of a basis.  We shall write $CW^{*}(E^1, E^2)$ for the morphism spaces in this category between two different objects, and $CW^{*}(E)$ for the endomorphism algebra of a single object.   The main technical result we shall prove is an extension of Theorem 1.1 of \cite{generate}:
\begin{thm} \label{thm:localising}
If $L$ resolves the diagonal, then trivial local systems over $L$ split-generate $\Seidel(M)$.  More precisely, trivial local systems over $L$ of a given dimension split-generate the subcategory of $\Seidel(M)  $ consisting of local systems of equal or lesser dimension.
\end{thm}
\begin{rem}
In his proof of Theorem \ref{thm:simply_connected} for odd dimensional spheres which appeared in \cite{damian}, Damian was first to consider Floer cohomology with coefficients in the infinite local system which corresponds to the group ring of $\pi_{1}(N)$.  This local system will also be used in our proof of  Theorem \ref{thm:simply_connected}.
\end{rem}
\begin{rem}
Except for Appendix \ref{sec:FSS} which proves the statement that the inclusion of $Q$ in $\TN$ induces an isomorphism on cohomology, this entire paper is written over the field with two elements, so that, unadorned, the symbols $\otimes$ and $\Hom$ will stand for the tensor product and maps of vector spaces over $\bF_{2}$, and all cochain and cohomology groups, whether coming from Floer theory or classical constructions are taken with such coefficients. The only modification that would be required  for our statements to hold over the integers is the requirement that the Lagrangians be relatively spin, together with an appropriate sprinkling of signs.  Such a modification is not required for our purpose.
\end{rem}
In the finite rank case, Theorem \ref{thm:localising} asserts that a local system of a given finite rank lies in the category split-generated by the trivial local system on $L$ of equal rank.  However, since any trivial local system of finite rank is quasi-isomorphic to a (finite) direct sum of copies of $L$ equipped with the trivial local system of rank $1$, which is the object of $\Seidel(M)$ corresponding to the Lagrangian $L$.  We conclude  that $L$ itself split-generates the subcategory of $ \Seidel(M) $ consisting of local systems whose dimension is finite, which gives a minor strengthening of the main result proved  in \cite{generate}.

If the dimension of a trivial system is some infinite cardinal number, then it is not in general quasi-isomorphic to a direct sum of copies of $L$, and the functor
\begin{equation}
 \xymatrix@C40pt{ \Seidel(M) \ar[r]^>>>>>>>>>{CW^{*}(L, \_ )} &   \mod-CW^{*}(L) }
\end{equation}
may not be fully faithful as the following example shows:
\begin{example}
  Consider $M=T^* S^1$, and let $L$ be the cotangent fibre at a point, which resolves the diagonal by the results of \cite{fibre-generate}, and whose Floer cochain algebra  is isomorphic to the group ring of the integers.  This group ring also gives rise to a local system   $\bF_{2}[t,t^{-1}]$ on $S^1$ whose monodromy around a counter-clockwise loop is multiplication by $t$.   The reader may easily check that  the Floer cohomology from $L$ to this local system is again isomorphic to the group ring as a right module.  If $CW^{*}(L, \_)$ were a fully faithful embedding, the fibre would be quasi-isomorphic  to this local system, which contradicts the fact that the Floer cohomology in the other direction is the dual module $\Hom( \bF_{2}[t,t^{-1}] , \bF_{2}  ) $.   Note that the only fact we used here is the  contractibility of the universal cover, so that one can perform the same computation for any aspherical manifold.
\end{example}

The proof of Theorem \ref{thm:localising} is given in Section \ref{sec:bimod-twist-compl}, where it is reduced to Lemma \ref{lem:evaluation_is_composition} and Proposition \ref{prop:commutative_diagram} which respectively assert the existence of a certain map of bimodules, and the commutativity of a diagram.  Once the correct extension of the wrapped Fukaya category has been constructed (as we shall do in the next section), one must adapt both the algebraic and geometric methods developed in \cite{generate} while keeping track of local systems throughout.  The algebraic part is discussed in Section \ref{sec:bimod-twist-compl}, and requires a bit of work for local systems supported on non-closed Lagrangians.  The geometric part of the construction is significantly easier, and we shall explain it in Section \ref{sec:constr-struct-maps}.

The proof of Theorem \ref{thm:simply_connected}  is relegated to Section \ref{sec:fukaya-categ-cover}, but we shall use the remainder of this introduction to indicate the ideas.  The main point is that any local system $E$ on a closed exact Lagrangian $Q \subset \TN$ defines a local system on $N$ whose fibres are the Floer cohomology groups $HW^{*}(\Tn, E) $ for varying fibres.  One of the strategies pursued by Fukaya and Smith in order to prove the results that appear in \cite{FSS2} was to produce such a local system using what is called \emph{Family Floer cohomology} by choosing perturbations that locally eliminate all singularities of the projection $Q \to N$.

Instead, we use homological algebra to produce such a local system (see, in particular Lemma \ref{lem:E-iso-complex-local-systems}), and Theorem \ref{thm:localising}  to show that this functor from local systems on $Q$ to local systems on $N$ is a cohomologically fully faithful embedding.  More precisely, applying  Theorem \ref{thm:localising}, together with the proof that a cotangent fibre resolves the diagonal,  implies that we obtain such a cohomologically fully faithful embedding by considering categories of such local systems equipped with morphisms defined using Floer cohomology.  Via the correspondence between Floer and ordinary cohomology (see Appendix \ref{sec:floer-morse-class}), we conclude that the category of local systems on $Q$ whose morphism spaces are
\begin{equation}
  \Hom(E^1,E^2) = H^{*}(Q, \Hom(E^1,E^2)),
\end{equation}
also embeds in the category of local systems on $N$.

If $N$ is simply connected, the proof is essentially complete:  the only local systems on $N$ are trivial, and this category is not rich enough to admit an embedding from the category of local systems on a non-simply connected manifold.  The precise result that we prove is given in Lemma \ref{lem:local_system_iso_system_on_base}.  This shows that every closed exact Lagrangian in a simply connected cotangent bundle is also simply connected.

\begin{rem}
Whenever $N$ is not simply connected, the proof of Theorem \ref{thm:simply_connected} will require constructing a Fukaya category associated to the universal cover of $\TN$. Even when this cover has finite type, the category  we shall  construct differs from the usual wrapped Fukaya category; this can be seen most easily by noting that the universal cover of $T^* S^1$ is symplectomorphic to the plane $\bR^2$ with its standard symplectic structure, and all Lagrangians in $\bR^2$  have vanishing wrapped Floer cohomology groups.  This later statement is familiar for closed Lagrangians as they may be displaced by Hamiltonian isotopies, but also holds for non-closed Lagrangians as a consequence of the vanishing of symplectic cohomology.  On the other hand, the category we shall assign to $T^* \bR$ will at least have the fibre and the zero section as non-vanishing objects. 

The model we shall use is related to ideas that have appeared in coarse geometry, in particular the notion of \emph{finite propagation} (see e.g. \cite{roe} and Example \ref{ex:all-fibres}). 
\end{rem}

\subsection*{Notation and conventions}
There are no new moduli spaces introduced in this paper which have not already appeared in \cite{generate}.  All new ideas involve using these moduli spaces for constructions which have a more infinite flavour.  Therefore, most details about the  families of Cauchy-Riemann operators satisfying the correct properties are suppressed, and we focus  on defining the desired algebraic structures from moduli spaces of curves.  In particular, the expression \emph{pseudo-holomorphic curve} stands for a solution to the appropriate Cauchy-Riemann operator.

\subsection*{Acknowledgments}
Comments by Shmuel Weinberger and Kate Ponto during my early attempts to prove these results helped me understand that there should be a Fukaya category associated to a universal cover which would allow the correct proof of Theorem \ref{thm:simply_connected} for a  simply connected base to extend almost immediately to the general case.   

While trying to find such a proof, the algebraic aspects of this paper evolved significantly.  I am most grateful to Andrew Blumberg, Amnon Neeman, and Dima Orlov for answering numerous questions which helped me correct various misconceptions.  I would like to thank Ivan Smith for many conversations we had about this, and related problems, and for helpful comments on an early draft.  Finally, patient comments from anonymous referees have been greatly helpful in producing a conceptually clearer proof, and improving the exposition.

\section{Extending the Fukaya category}
Recall that a Liouville manifold is a smooth manifold  $M$ equipped with a $1$-form $\lambda$ whose differential $\omega$ is a symplectic form and which satisfies the following additional property:   there exists a codimension $0$ compact submanifold with boundary $M^{in} \subset M$ such that $\lambda$ restricts to a contact form on $\partial M^{in} $ and the complement admits a diffeomorphism
\begin{equation} \label{eq:infinite_end_contact}
M -  \inte(M^{in}) \cong   [1, +\infty) \times \partial M^{in},
\end{equation}
which takes $\lambda$ to $r (\lambda| \partial M^{in} ) $  where $r$ is the coordinate on $[1,+\infty) $.  We shall say that $  [1, +\infty) \times \partial M^{in} $  is the infinite end of $M$, and write $\psi^{\rho}$ for the time-$\log(\rho)$ Liouville flow on $M$.   We assume that the Reeb flow associated to $\lambda$ is generic is the sense that all Reeb orbits are non-degenerate; this can be achieved by a small perturbation of $\partial M^{in}$.

The geometric Lagrangians we shall consider are exact Lagrangians $L \subset M$ such that the following additional condition holds: If we write $L^{in}  $ for the intersection of $L$ with $M^{in}$, then
 \begin{equation} \label{eq:boundary_Legendrian} \parbox{35em}{$\partial L^{in}$ is Legendrian, and the complement $L - L^{in}$ is given by the product $[1,+\infty) \times \partial L^{in}  $ in the coordinates of Equation \eqref{eq:infinite_end_contact}.} \end{equation}
We may equivalently require that $\lambda$ vanish away from $L^{in}$.   In addition, we also assume that all Reeb chords with endpoints on $\partial L^{in}$ are non-degenerate, which can be achieved by a small perturbation of $L$, preserving the Legendrian boundary condition.

\subsubsection{Auxiliary choices} 
We fix a generic Hamiltonian function $H \co M \to \bR$ which agrees with $r^2$ along the infinite end of $M$.  Given a pair of Lagrangians $(L^0, L^1)$, we define
\begin{equation}
  \label{eq:chords}
  \Chord(L^0, L^1)
\end{equation}
to be the set of time-$1$ Hamiltonian flow lines of $H$ which start on $L^0$ and end on $L^1$.   These are maps
\begin{equation}
  x \co [0,1] \to M
\end{equation}
whose tangent vector is the Hamiltonian vector field associated to $H$, and which take $0$ to $L^0$ and $1$ to $L^1$.   In particular, if $L^0=L^1$, every critical point of the restriction of $H$ gives such a chord, though there might be others.

We choose a primitive for the restriction of $\lambda$ to every Lagrangian.  If  $f^0$ and $f^1$ are respectively primitives on $L^0$ and $L^1$, we then assign a real-valued \emph{action} to each chord $x \in \Chord(L^0, L^1)$:
\begin{equation}
  \Action(x) = \int - x^{*}(\lambda) + H \circ x \, dt + f^1(x(1)) - f^{0}(x(0)). 
\end{equation}
Our geometric conditions imposed on Lagrangians (and on the Hamiltonian) imply
\begin{lem}[See Lemma 3.1 of \cite{fibre-generate}] \label{lem:properness_chords}
For each real number $a$, the set $  \Chord_{\geq a}(L^0, L^1) $  of chords whose action is bounded below by $a$ is compact.  For a generic $C^{\infty}$ small perturbation, of either $L^0$ or $L^1$, the sets $  \Chord_{\geq a}(L^0, L^1) $  are finite.  \qed
\end{lem}
\begin{rem}
  The reader might want to note that our conventions on actions differ by a sign from those used, for example, by Abbondandolo and Schwarz in \cite{ASchwarz}.  With the choices we make, the differential strictly \emph{raises action}, and all other operations raise action up to a possible additive constant coming from the inhomogeneous nature of the Cauchy-Riemann equation we use.  
\end{rem}

In order to enhance Floer cohomology to a theory defined over integrally graded complexes, we first assume that  $ 2 c_{1}(M) $ vanishes, and consequently fix a quadratic complex volume form $\eta$ on $M$.  Given a Lagrangian  $L \subset M$ we obtain a phase map
\begin{equation}
  L \to S^{1}
\end{equation}
whose value at a point $x$ is obtained by evaluating $\eta/|\eta|$ on a basis of $T_{x}L$.   Since $\eta$ is a quadratic volume form,  its value on a basis does not change upon reordering the elements, so that this map is well defined, and does not, in particular, depend on choosing an orientation of $L$.

The vanishing of the associated cohomology class in $H^{1}(L,\bZ)$ is the obstruction to $L$ defining an object of the $\bZ$-graded Fukaya category of $M$ with respect to the chosen volume form.  Whenever the phase map defines a trivial cohomology class, it factors through $\bR$.  Lagrangians equipped with such a factorisation are called \emph{graded} (see Section 12 of \cite{seidel-book}).   Choosing a grading on both $L_0$ and $L_1$ assigns a \emph{Maslov index} which we denote $|x|$ to each chord $x \in \Chord(L^0, L^1)$. 

\subsection{Wrapped Floer cohomology} \label{sec:wrapp-floer-cohom}

\begin{defin}
The objects of $\Seidel(M)$ are pairs $(E,L)$ where $E$ is a local system of chain complexes on an exact graded Lagrangian  $L$ satisfying Conditions \eqref{eq:boundary_Legendrian}.
\end{defin}
\begin{rem}
Recall that a local system on $L$ is the assignment of a vector space $E_{x}$ to every point $x \in L$ (the fibre at $x$), and of a \emph{parallel transport} map
\begin{equation}
E_{x} \to E_{y}
\end{equation}
to every homotopy class of paths starting at $x$ and ending at $y$, such that the map associated to the concatenation of two paths (with a common endpoint) agrees with the composition of their associated parallel transport maps.  Two local systems are isomorphic if there are isomorphisms of all fibres which commute with parallel transport maps, and the set of local systems up to isomorphism agrees with the set of representations of the fundamental group of $L$.  The isomorphism between these two sets depends on choosing a basepoint $x \in L$, and the map in one direction assigns to a local system its \emph{monodromy representation}
\begin{equation}
  \pi_{1}(L,x) \to \End(E_{x}).
\end{equation}
A local system of complexes is a collection $E_{i}$ of local systems indexed by the integers, together with maps of local systems $\delta \co E_{i} \to E_{i+1}$ which square to $0$.  We say that such a complex is bounded if $E_{i}$ vanishes whenever $i \ll 0$ and $0 \ll i$.
\end{rem}
Let $(E^{0}, L^0)$ and $(E^{1}, L^{1})$ be two complexes of local systems.  We shall define a  Floer complex
\begin{equation}
  CW^{*}(E^{0}, E^{1})
\end{equation}
which generalises the usual wrapped Floer complex that we would obtain by specialising to the case each local system $E^i$ is trivial of rank one.  As a graded vector space we have
\begin{equation} \label{eq:wrapped_complex}
  CW^{k}(E^{0}, E^{1})  = \bigoplus_{x \in \Chord(L^0, L^1)}  \Hom^{k-\deg(x)}(  E_{x(0)}^{0},   E_{x(1)}^{1}).
\end{equation}
In the right hand side, $ \Hom^{k-\deg(x)}(  E_{x(0)}^{0},   E_{x(1)}^{1}) $ is the space of linear maps from $  E_{x(0)}^{0}$ to $ E_{x(1)}^{1} $ of degree $ k-\deg(x) $.

One particularly useful example will come from studying trivial local systems.  Given a chain complex $V$ (of arbitrary dimension), we write $V \boxtimes L$ for the trivial local system over $L$ with fibre $V$, and omit $V$ altogether when it has rank $1$.  It is useful to record some immediate consequences of this definition in the presence of trivial local systems:
\begin{lem} \label{lem:factoring_trivial_complexes}
For any pair of chain complexes $U_0$ and $U_1$, and Lagrangians $L^0$ and $L^1$, we have an isomorphism of graded vector spaces
\begin{equation}
  CW^{*}(U_0 \boxtimes L^0, U_1 \boxtimes L^1)   = \Hom(U_0,U_1) \otimes CW^{*}(L^0,L^1).
\end{equation}
Assuming that $E^0$ is trivial  with fibre $U_0$, and either that $L^1$ is compact or that $U^0$ is finite dimensional, we have
\begin{equation}
  CW^{*}(U_0 \boxtimes L^0, E^1)   = \Hom\left(U_0,CW^{*}(L^0,E^1) \right),
\end{equation}
while assuming that $E^1$ is trivial with fibre $U_1$, and that either $E^0$ or $U_1$ is finite dimensional, we have
\begin{equation}
  CW^{*}(E^0, U_1 \boxtimes L^1 )   =CW^{*}(E^0,  L^1 ) \otimes U_1.
\end{equation}
\end{lem}
\begin{proof}
We describe the second case, and leave the remaining ones to the reader.  Whenever $U_{0} \boxtimes L^0$ is a finite-dimensional trivial local system over $L^0$ with fibre $U_0$,  
\begin{align*}
  CW^{*}(U_{0} \boxtimes L^0 , E^1) & \cong \bigoplus_{x \in \Chord(L^0, L^1)  } \Hom( U_0 ,  E^{1}_{x(1)}) \\
& \cong \Hom \left( U_0 , \bigoplus_{x \in \Chord(L^0, L^1)  } \Hom( \bF_{2} ,  E^{1}_{x(1)})  \right)\\
& \cong \Hom \left( U_0 , CW^{*}(L^0,E^1) \right).
\end{align*}
Also, whenever $L^0$ is closed, there can only be finitely many chords with endpoints on $L^0$ and $L^1$, so the direct sum in Equation \eqref{eq:wrapped_complex} is finite.  In particular, allowing $U_0$  to have arbitrary dimension we still have
\begin{align*}
  CW^{*}( U_{0} \boxtimes L^0 , E^1) &  \cong \Hom(U_0 , CW^{*}(L^0,E^1) ).
\end{align*}
\end{proof}

\subsection{The differential}\label{sec:diff}
\begin{figure}
  \centering
\includegraphics{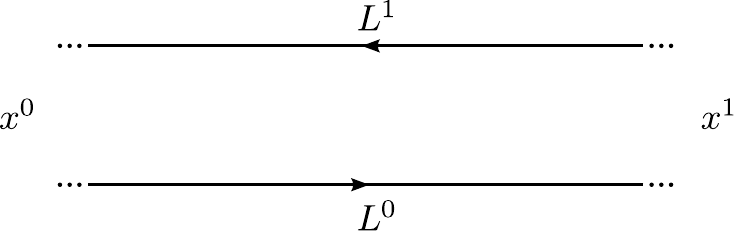}
  \caption{ }
  \label{fig:strip}
\end{figure}
Given two chords $x^{0}$ and $x^{1}$ starting on $L^0$ and ending on $L^1$, we set $\Disc(x^0, x^1)$ to be the moduli space of  strips $u$ with boundary conditions on $L^0$ and $L^1$ and converging at $-\infty$ to $x^0$ and at $+\infty$ to $x^1$ (see Figure \ref{fig:strip}), solving the Equation
\begin{equation} \label{eq:CR-equation-strip}
  (du - X_{H} \otimes dt)^{0,1} = 0
\end{equation}
with respect to some family of compatible almost complex structures on $M$.  In order to ensure that the operations we define using these moduli spaces respect action filtrations, we need the following result:
\begin{lem} \label{lem:differential_respects_filtration}
  If $ \Disc(x^0, x^1) $ is not empty, then
  \begin{equation}
    \Action(x^0) \geq \Action(x^1).
  \end{equation}
\end{lem}
\begin{proof}
We define the energy of $u$ to be
\begin{equation}
  E(u) = \int |du - X_{H} \otimes dt|^{2}
\end{equation}
where the norm is taken with respect to the metric obtained from the symplectic form and the complex structure which enters in the Cauchy-Riemann equation \eqref{eq:CR-equation-strip}.  In particular, we find that
\begin{equation}
 0  \leq E(u) = \int u^{*}(\omega) - d (H \circ u) \otimes dt.
\end{equation}
Applying Stokes's theorem, we conclude that 
\begin{multline}
   0 \leq   \int - {x^{0}}^{*}(\lambda) + H \circ x^{0} dt -  \int - {x^{1}}^{*}(\lambda) + H\circ x^{1} dt \\
+ \int_{\bR \times \{ 0 ,1\}} u^{*}(\lambda) - H \circ u \otimes dt|\bR \times \{0,1\} .
\end{multline}
At this stage, we use the fact that (i) the restriction of $dt$ to the boundary of the strip vanishes and (ii) the images of  $ u (  \bR \times \{0\} ) $ and $ u (  \bR \times \{1\} ) $ respectively lie on $L^{0}$ and $L^{1} $ and we have chosen primitives for the restriction of $\lambda$ to these exact Lagrangians.  We compute that
   \begin{multline}
   0 \leq   \int - {x^{0}}^{*}(\lambda) + H \circ x^{0} dt -  \int - {x^{1}}^{*}(\lambda) + H\circ x^{1} dt \\
+ f^{1}(x^{0}(1)) - f^{1}(x^{1}(1)) - f^{0}(x^{0}(0)) + f^{0}(x^{1}(0)).
\end{multline}
Rearranging the terms, we conclude, as desired that
\begin{equation}
    0 \leq  \Action(x^{0}) - \Action(x^{1}).
\end{equation}
\end{proof}

If $u$ is a strip in  $\Disc(x^0, x^1)$, we define
\begin{align*}
  \gamma_{u}^{0} \co E^{0}_{x^0(0)} & \to E^{0}_{x^1(0)} \\ 
  \gamma_{u}^{1} \co E^{1}_{x^1(1)} & \to E^{1}_{x^0(1)} 
\end{align*}
to be the parallel transport maps along the images under $u$ of the two boundary components.  Note the different directions of the parallel transport maps which are indicated by the arrows in Figure  \ref{fig:strip}. We then associate to $u$ a linear map
\begin{align} \label{eq:transport_map_strip}
  \mu^{u} \co  \Hom( E^{0}_{x^1(0)},   E^{1}_{x^1(1)}) & \to \Hom( E^{0}_{x^0(0)},   E^{1}_{x^0(1)}) \\
\mu^{u}(\phi)(a) & =  \gamma_{u}^{1}  \circ \phi \circ  \gamma_{u}^{0} (a).
\end{align}
To decode this formula, start with $a \in E^{0}_{x^0(0)}$  together with a linear map $ \phi \in   \Hom( E^{0}_{x^1(0)},   E^{1}_{x^1(1)}) $.  We first move $a$ using parallel transport to obtain the element $  \gamma_{u}^{0} (a) $  in $  E^{0}_{x^1(0)} $.  We then apply $\phi$, and move the result back to the fibre of $E^1$ over $ x^0(1) $ using parallel transport.  

In addition, since both $E^{0}$ and $E^{1}$ have an internal differential $\delta_0$ and $\delta_1$, we have a map
\begin{align}
   \delta \co  \Hom( E^{0}_{x(0)},   E^{1}_{x(1)}) & \to \Hom( E^{0}_{x(0)},   E^{1}_{x(1)}) \\
\delta(\phi) & = \phi \circ \delta_{0} + \delta_{1} \circ \phi.
\end{align}

 The differential $\mu^1$ on $   CW^{k}(E^0, E^1) $ is defined as a sum of this internal differential with the contributions $\mu^{u} $ associated to each rigid strip
\begin{align}
 \notag \mu^{1} \co   CW^{k}(E^0, E^1)    & \to  CW^{k+1}(E^0, E^1)    \\
\mu^{1}  & = \delta +  \sum_{ \substack{u \in \Disc(x^0, x^1) \\ k = |x^1| = |x^0|-1}  }  \mu^{u}. \label{eq:differentia_sum}
\end{align}

\begin{lem} \label{lem:differential_squares_0}
The map $\mu^1$ is a differential.
\end{lem}
\begin{proof}
Gromov compactness implies  that the sum in Equation \eqref{eq:differentia_sum} is finite, and hence that $\mu^1$ is well-defined.  To prove that the differential squares to $0$,  we note that since $E^0$ and $E^1$ are local systems,  the homotopy class of a path determines the associated parallel transport map.   Moreover, any two pseudo-holomorphic strips which lie in the same component of a moduli space $\Disc(x,y)$  restrict to homotopic paths on the boundary.  Passing to the Gromov compactification, we find that any two broken strips which lie on the same component must have the property that the paths obtained by restricting them to the boundary, then concatenating, are homotopic.  In particular, we may adapt the usual proof in Floer theory which uses the fact that the Gromov compactification of the moduli spaces of dimension $1$ are closed intervals whose boundary points may be identified with the terms in the composition $( \mu^1 - \delta) \circ (\mu^1 - \delta)$, which then necessarily vanishes.  Since both $\delta_0$ and $\delta_1$ square to $0$, so does $\delta$.  Moreover, since $\delta_0$ and $\delta_1$ commute with parallel transport, $\delta$ commutes with $  \mu^1 - \delta$.  Having accounted for all terms in $\mu^{1} \circ \mu^{1}$, we conclude that it vanishes.   
\end{proof}

Having defined the differential, we can now revisit our computation of Floer complexes in the presence of trivial local systems:
\begin{lem}  \label{lem:iso_complexes_trivial_V}
Under the hypotheses stated in Lemma \ref{lem:factoring_trivial_complexes}, we have isomorphisms of complexes
\begin{align} \label{eq:end(V)-HF(L)}
   CW^{*}(U_0 \boxtimes L^0, U_1 \boxtimes L^1) &  = \Hom(U_0,U_1) \otimes CW^{*}(L^0,L^1) \\ \label{eq:factor_U_0}
  CW^{*}(U_0 \boxtimes L^0, E^1) &  = \Hom\left(U_0,CW^{*}(L^0,E^1) \right) \\
  CW^{*}(E^0, U_1 \boxtimes L^1 )  &  = CW^{*}(E^0,  L^1 ) \otimes U_1. \label{eq:compact_object}
\end{align}
\end{lem}
\begin{proof}
Consider the case when $E^0$ is assumed to be trivial and $U_0$ to be of finite rank.  Identifying all fibres of $E^0$ with $U_0$ in this case,  the map in Equation \eqref{eq:transport_map_strip}, which defines the differential on the left hand side of Equation \eqref{eq:factor_U_0}, becomes
\begin{align} 
  \mu^{u}_{E^0,E^1} \co  \Hom( U_0,   E^{1}_{x^1(1)}) & \to \Hom( U_0,   E^{1}_{x^0(1)}) \\
\mu^{u}_{E^0,E^1}(\phi)(a) & =   \gamma_{u}^{1} ( \phi(a)).
\end{align}
Note that this is the definition of the differential on the right hand side of Equation \eqref{eq:factor_U_0}, which proves the isomorphism between these chain complexes.   All other cases follow from a similar analysis.
\end{proof}

\begin{rem} \label{rem:compactness}
If, in Equation \eqref{eq:compact_object}, we assume that $E^0$ and  $E^1$ are supported on the same compact Lagrangian, we may interpret this formula to say that finite rank local systems on a given compact Lagrangian form \emph{compact objects} of the subcategory of $\Seidel(M)$ consisting of local systems with the same support. It seems quite likely that every finite rank local system over a Lagrangian (closed or not) defines a compact object in an appropriate geometric enlargement of $\Seidel(M)$ which would admit all coproducts.

In particular, the failure of $CW^{*}(L,\_) $  to be an embedding when $L$ resolves the diagonal suggests that  such an enlargement would not be \emph{compactly generated}.  It would be interesting to know whether it is well-generated in the sense of Neeman (see \cite{neeman}).
\end{rem}

\subsection{The $A_{\infty}$ structure}
We shall write $\Disc_{d}$ for the abstract moduli space of holomorphic discs with $d$ positive punctures $(\xi^1, \ldots, \xi^{d})$ and $1$ negative puncture which we denote $\xi^{0}$ or $\xi^{d+1}$ depending on the context.  We let $\Discbar_{d}$ denote its Deligne-Mumford compactification.  Let $L^0, \ldots, L^d$ be a sequence of exact Lagrangians in $M$.  We assume that the restriction of $H$ to each Lagrangian $L^{k}$ is a Morse function, and that all element of $ \Chord(L^{j}, L^{k}) $ are non-degenerate chords (note that the first condition implies the second condition can be achieved if $L^{j} = L^{k}$).   Given a sequence $\vx = (x^1, \ldots, x^d)$ of Hamiltonian chords such that $x^{k} \in \Chord(L^{k-1}, L^{k})$ as well as a chord $ x^{0} \in \Chord(L^{0}, L^{d}) $ we have a moduli space 
\begin{equation}
  \Disc(x^0, \vx)
\end{equation}
consisting of pseudo-holomorphic maps $u \co S \to M$ from a disc $S \in \Disc_{d}$.  In the usual definition of Lagrangian Floer cohomology, one requires that $u$ map the boundary components of $S$ to the Lagrangians $L^k$.  However, in order to construct Floer cohomology in the wrapped setting using a quadratic Hamiltonian, it is more convenient to study a moduli space of pseudo-holomorphic curves with \emph{moving Lagrangian boundary conditions} in which there exists a family of Lagrangians $L^{k}_{z,S}$ such that the boundary condition imposed on elements  $u\co S \to M $ of $  \Disc(x^0, \vx)  $ is:
\begin{equation}
  u(z) \in L^{k}_{z,S} \textrm{  if $z$  lies the interval between $\xi^{k}$ and $\xi^{k+1}$.}
\end{equation}
By construction (see Definition 4.1 and Equation (4.3) of \cite{generate}), we can always compose $u$ with a $z$-dependent diffeomorphism of $M$ so that the boundary conditions are indeed on $L^k$:
\begin{lem} \label{lem:diffeo_for_evaluation}
For each nodal surface $S \in \Discbar_{d}$, there exists a family of maps $\psi_{z,S} \co M \to M$ parametrised by $z \in \partial S$ such that, if $0 \leq k \leq d$,  then
\begin{align}
\psi_{z,S} \circ u (z) & \in L^{k} \textrm{  if $z$  lies the interval between $\xi^{k}$ and $\xi^{k+1}$} \\ 
  \lim_{z \to \xi^{k}} \psi_{z,S} \circ u (z) & = x^{k}(1)  \\
  \lim_{z \to \xi^{k+1}} \psi_{z,S} \circ u (z) & = x^{k+1}(0).
\end{align}
Moreover, if $S$ is decomposes into components $S_1$ and $S_2$, with $u_1$ and $u_2$ the corresponding maps on $S_1$ and $S_2$ then 
\begin{equation}
\left( \psi_{z,S} \circ u \right) |{S_i} = \psi_{z,S_i} \circ u_{i}.
\end{equation}
 \qed
\end{lem}
By considering parallel transport maps along $  \psi_{z,S} \circ u  $  we conclude:
\begin{cor}  \label{cor:parallel_transport_from_curve}
To each element $u \in  \Disc(x^0, \vx)$ there is a canonically assigned collection of parallel transport maps
 \begin{align}
  \gamma_{u}^{0} \co E^{0}_{x^0(0)} & \to E^{0}_{x^1(0)} \\ 
  \gamma_{u}^{k} \co E^{k}_{x^k(1)} & \to E^{k}_{x^{k+1}(0)} \textrm{ if  $1 \leq k \leq d-1$}  \\ 
  \gamma_{u}^{d} \co E^{d}_{x^d(1)} & \to E^{d}_{x^0(1)} .
\end{align} \qed
\end{cor}
\begin{rem}
Note that only the homotopy class of the restriction of $\psi_{z,S}$ to the union of all Lagrangians under consideration enters our construction.  In particular, it is not necessary to know that $ \psi_{z,S} $ are canonically determined, but only that the homotopy class of the restriction is compatible with breakings of holomorphic curves.   
\end{rem}

\begin{figure}
  \centering
\includegraphics{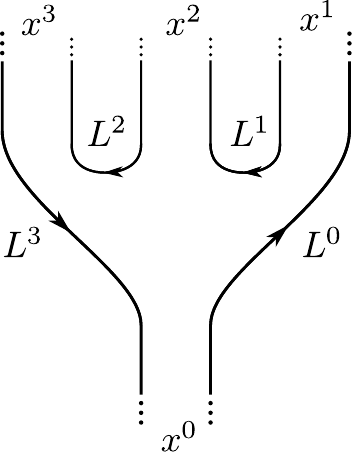}
  \caption{ }
  \label{fig:product}
\end{figure}
We now associate to each curve $u$ its contribution to the $A_{\infty}$ operations, which is a map
\begin{multline}
\mu^{u} \co  \Hom( E^{d-1}_{x^d(0)},  E^{d}_{x^d(1)}) \otimes \cdots \otimes   \Hom( E^{k-1}_{x^k(0)},  E^{k}_{x^k(1)})  \otimes \cdots \otimes  \Hom( E^{0}_{x^1(0)},  E^{1}_{x^1(1)}) \\ \to  \Hom( E^{0}_{x^0(0)}, E^{1}_{x^0(1)})
 \end{multline}
defined as follows (see Figure \ref{fig:product} for the case $d=3$): Given a tensor product of linear maps $\phi^{d}\otimes \cdots \otimes \phi^1 $ in the source, and an element $e$ in the fibre of $ E^{0} $  at the starting point of $x^0$, we first use parallel transport along the boundary of $\psi_{z} \circ u $ to move $e$ to the fibre at the starting point of the chord $x^1$.  We then apply the map $\phi^1$ to obtain an element of the fibre of $E^1$ at the endpoint of $x^1$, which we then transport along the boundary of our map to the starting point of the chord $x^2$.  Repeating this procedure, we end up with an element of the fibre of $E^d$ at the endpoint of $x^{d}$, which we can move using parallel transport to the end point of $x^0$.    The formula is
\begin{equation} \label{eq:higher_product_curve}
  \mu^{u}(\phi^d, \ldots, \phi^{1})(a) = \gamma_{u}^{d} \circ \phi^{d} \circ \cdots \circ \gamma_{u}^{1} \circ \phi^1 \circ \gamma_{u}^{0}(a).
\end{equation}
Again, Gromov compactness implies that given a sequence  $ (\phi^d, \ldots, \phi^{1}) $, there are only finitely many maps $u$ which are rigid (see, e.g. Lemma 3.2  of \cite{generate} for the case $d=2$).  In particular we have a well defined finite sum
\begin{align} \label{eq:A_oo_structure_local_systems}
  \mu^{d} \co CW^{*}(E^{d-1}, E^{d}) \otimes \cdots \otimes CW^{*}(E^{0}, E^{1}) & \to  CW^{*}(E^{0}, E^{d} ) \\
\mu^{d}  & = \sum_{\substack{u \in \Disc(x^0, \vx) \\  u \textrm{ is rigid} } }  \mu^{u},
\end{align}
which is the $d$\th higher product of the $A_{\infty}$ structure on $\Seidel(M)$.

We now derive the $A_{\infty}$ analogue of Lemma \ref{lem:factoring_trivial_complexes}, and  consider a trivial local system $V$ on a Lagrangian $L$.  Equation \eqref{eq:end(V)-HF(L)} asserts that we have an isomorphism of complexes
\begin{equation} \label{eq:end(V)-factors_out}
  CW^{*}(V \boxtimes L )   = \End(V) \otimes CW^{*}(L).
\end{equation}
In order to compute the $A_{\infty}$ structure on the left hand side, we use the fact that all parallel transport maps are the identity to rewrite Equation \eqref{eq:higher_product_curve} as
\begin{equation}
   \mu^{u}(\phi^d, \ldots, \phi^{1})(a) = \phi^{d} \circ \cdots \circ \phi^1 (a).
\end{equation}
Taking the sum over all discs $u$, and writing  $x_{i}$  for the generator of $\Hom(\bF_{2}|x_{i},  \bF_{2}|x_{i}) $, we find that the $A_{\infty}$ structure on $   CW^{*}(V \boxtimes L )  $ is given by
\begin{equation} 
\label{eq:higher_products_tensor}
  \mu^{d}( \phi^d \otimes x_d,  \ldots, \phi^1 \otimes x_1) = \left( \phi^{d} \circ \phi^{d-1}\circ \cdots \circ  \phi^{1}\right) \otimes  \mu^{d}(x_d,  \ldots, x_1). 
\end{equation}
This is precisely the formula for the $A_{\infty}$ structure on the tensor product of the ordinary algebra  $\End(V)$ with  the $A_{\infty}$ algebra $ CW^{*}(L) $.

More generally, consider an object $E$ of $\Seidel(M)$ which is supported on a compact Lagrangian.    From Equation \eqref{eq:factor_U_0}, we have an isomorphism of complexes
\begin{equation} \label{eq:iso_of-modules-tensor_V}
  CW^{*}(V \boxtimes L, E )   = \Hom\left(V, CW^{*}(L,E) \right).
\end{equation}
Note that the right hand-side is naturally a right $A_{\infty}$-module over $\End(V) \otimes CW^{*}(L)  $, with operations
\begin{equation} \label{eq:structure_tensor_with_v}
  \mu^{1|d}(  \phi,  \phi^d \otimes x_d,  \ldots, \phi^1 \otimes x_1)(v) = \mu^{d+1}( \phi^{d} \circ \phi^{d-1}\circ \cdots \circ  \phi^{1} \circ \phi (v), x_d,  \ldots, x_1).
\end{equation}
The same analysis as for the algebra structure on $  CW^{*}(V \boxtimes L )$ shows that these structure maps are the same as those coming from considering $CW^{*}(V \boxtimes L, E )  $ as a module over $CW^{*}(V \boxtimes L )  $ using the $A_{\infty}$ structure of $\Seidel(M)$ given by Equation \eqref{eq:A_oo_structure_local_systems}. Summarising these results, we conclude:
\begin{lem} \label{lem:Floer_infinity-structure-cpatible-with-tensor}
Identifying $   CW^{*}(V \boxtimes L ) $ and $\End(V) \otimes CW^{*}(L)  $ as $A_{\infty}$ algebras, the right $A_{\infty}$ module structures on $   CW^{*}(V \boxtimes L, E )  $ and $ \Hom\left(V, CW^{*}(L,E) \right) $ are isomorphic. \qed 
\end{lem}

We end this section by discussing action filtrations.  First, let us introduce the notation
\begin{equation}
  \left( CW^{*}(E^{d-1}, E^{d}) \otimes \cdots \otimes CW^{*}(E^{0}, E^{1})  \right)_{\geq c}
\end{equation}
for the subspace generated by all tensor products supported by chords the sum of whose actions is bounded below by some constant $c \in \bR$.  Lemma \ref{lem:differential_respects_filtration} implies that this is in fact a subcomplex.  While this subcomplex may be infinitely generated, only generators corresponding to finitely many chords appear in it, as follows from Lemma \ref{lem:properness_chords}.  In particular, the following result is a consequence of  the finiteness of the set of holomorphic discs with a given set of inputs (see, e.g. Lemma 3.2  of \cite{generate} which considers the case of the product):
\begin{lem} \label{lem:action_filtration_infinity}
For each constant $c$ (and each collection of Lagrangians $L^0, \ldots, L^{d}$), there exists some constant $b$ such that the image of
\begin{equation}
  \mu^{d}  \co \left( CW^{*}(E^{d-1}, E^{d}) \otimes \cdots \otimes CW^{*}(E^{0}, E^{1}) \right)_{\geq c}  \to  CW^{*}(E^{0}, E^{d} ) 
\end{equation}
lies in the subcomplex
\begin{equation}
   CW^{*}_{\geq b}(E^{0}, E^{d} )
\end{equation}
of morphisms whose action is bounded below by $b$. \qed
\end{lem}
\begin{rem}
In our statement of the previous result, we give no quantitative bound for $b$ in terms of $c$.  In order to provide such a bound, we would have to make explicit the fact that  the diffeomorphisms $ \psi_{z,S} $ are conformal symplectomorphisms, i.e. they rescale the symplectic form.  With this in mind, an analysis similar to that performed in Lemma \ref{lem:differential_respects_filtration} can be used to prove that the difference between $b$ and $c$ can be bounded in terms of the supremum of the primitive for $\lambda|L^i$, and the conformal constants of the maps $\psi_{z,S}$.
\end{rem}

\subsection{A generalisation of the product} \label{sec:gener-prod}
The following generalisation of the product will only be used for the proof of Theorem \ref{thm:simply_connected} in the non-simply connected case:  Let $E^{0,1}$, $E^{1,2}$ and $E^{0,2}$ be local systems on the same closed exact Lagrangian $Q$, and assume that we are given a map of local systems
\begin{equation}
  E^{1,2} \otimes E^{0,1} \to E^{0,2}.
\end{equation}
One can imitate the construction of $\mu^2$ to define a map
\begin{equation} \label{eq:product_local_sys}
  CW^{*}(Q, E^{1,2}) \otimes CW^{*}(Q, E^{0,1}) \to  CW^{*}(Q, E^{0,2} )
\end{equation}
with $Q$ denoting, as before, the trivial local system of rank $1$.  Start by assuming that the Hamiltonian $H$ is sufficiently $C^{2}$ small near $Q$ that all chords which start and end on $Q$ are constant, and map to critical points of $H$.   In this case, we have a canonical identification between the fibre of $E^{i,j}$ at the beginning and at the end points of such a chord, so we may describe the Floer complex in two different ways
\begin{equation} \label{eq:Floer_cohomology_local_system_coeff}
    CW^{*}(Q, E^{i,j})  = \bigoplus_{x \in \Chord(Q)}  E^{i,j}_{x(1)} = \bigoplus_{x \in \Chord(Q)}  E^{i,j}_{x(0)}.
\end{equation}
Given a holomorphic disc $u$ with two incoming ends converging to $x^1$ and $x^2$, and one outgoing end converging to $x^{0}$, we have parallel transport maps
 \begin{align}
  \gamma_{u}^{0,1} \co E^{0,1}_{x^1(0)} & \to E^{0,1}_{x^0(0)} = E^{0,1}_{x^0(1)} \\ 
    \gamma_{u}^{1,2} \co E^{1,2}_{x^2(1)} & \to E^{1,2}_{x^0(1)} .
\end{align}
The product in Equation \eqref{eq:product_local_sys} is defined via the formula
\begin{equation} \label{eq:product_local_sys_formula}
  \phi_{1,2} \otimes \phi_{0,1} \mapsto \sum_{u}  \left( \gamma_{u}^{1,2}  \phi_{1,2}\right)  \cdot \left( \gamma_{u}^{0,1}  \phi_{0,1}\right).
\end{equation}
 
There is a special case of this construction in which $E^{0}$, $E^{1}$, and $E^{2}$ are three local systems and $ E^{i,j} = \Hom( E^i, E^j) $.  By comparing Equation \eqref{eq:Floer_cohomology_local_system_coeff} with the definition of $CW^{*}(E^{i},E^{j})$, we find the same description as a direct sum.  Moreover, comparing the product formula in Equation \eqref{eq:product_local_sys_formula} with the one given in Equation \eqref{eq:higher_product_curve}, we see that exactly the same holomorphic curves are counted.  We summarise this discussion in the following result:
\begin{lem}
If $E^i$ and $E^j$ are local systems on a closed exact Lagrangian $Q$, then, for a generic Hamiltonian which is $C^{2}$ small near $Q$, there is an isomorphism of chain complexes
\begin{equation}
CW^{*}(E^{i}, E^{j}) =    CW^{*}(Q, \Hom(E^i,E^j)).
\end{equation}
 This isomorphism intertwines the product $\mu^{2}$ with the one defined by Equation \eqref{eq:product_local_sys_formula}. \qed
\end{lem}

\section{Simply connected cotangent bundles}
\label{sec:simply-conn-cotang}

The goal of this section is to prove Theorem \ref{thm:simply_connected} assuming that the total space is simply connected.  The only input that is needed from later sections is Theorem \ref{thm:localising}, which was already stated in the introduction, and the results appearing in the Appendices.  Moreover, we shall, without further mention, apply the homological perturbation lemma to pass from the $A_{\infty}$ structure defined by Floer cochains, to a minimal $A_{\infty}$ structure supported on cohomology.

We start by recalling, from \cite{ASchwarz}, that the wrapped Floer cohomology  of a cotangent fibre is isomorphic to the homology of the based loop space.  With our cohomological grading, this isomorphism is such that $HW^{*}(\Tn)$ is supported in non-positive degree.  The first result we state holds for all cotangent bundles:
\begin{lem} \label{lem:E-iso-complex-local-systems}
If $E$ is a local system of complexes on a closed exact Lagrangian, which is supported in finitely many degrees, then $E$ is quasi-isomorphic, in the category of twisted complexes over $\Seidel(\TN)$ to a twisted complex constructed from iterated cones of local systems $E^{k}_{N}$ on $N$ whose fibres are $HW^{k}( \Tn,  E ) $:
\begin{equation}
  \xymatrix{\cdots \ar[r] \ar@/_1pc/[rr] \ar@/_2pc/[rrr]  \ar@/_3pc/[rrrr]& E^{k+1}_{N} \ar[r]  \ar@/_1pc/[rr] \ar@/_2pc/[rrr]  & E^{k}_{N} \ar[r]   \ar@/_1pc/[rr]   & E^{k-1}_{N} \ar[r]   & \cdots\\
\\ }
\end{equation}
\end{lem}

\begin{proof}
The main result of \cite{fibre-generate} is that a cotangent fibre resolves the diagonal.  By Theorem \ref{thm:localising}, $E$ lies in the subcategory of $\Seidel(\TN)$ split-generated by $V \boxtimes \Tn$ for some vector space $V$.  Moreover, since $HW^{*}(  V \boxtimes \Tn)$ is supported in non-negative degrees, Lemma \ref{lem:filtration_degree} implies that the right module $HW^{*}(  V \boxtimes \Tn,  E )  $ admits a filtration whose subquotients are the cohomology groups
\begin{equation}
  HW^{k}(  V \boxtimes \Tn,  E)
\end{equation}
equipped with their natural module structure over $HW^{*}(  V \boxtimes \Tn) $.  Applying Lemma \ref{lem:Floer_infinity-structure-cpatible-with-tensor}, then passing to minimal models, we conclude that under the identification
\begin{equation}
  HW^{*}(  V \boxtimes \Tn)  = \End(V) \otimes  HW^{*}(\Tn)
\end{equation}
as $A_{\infty}$ algebras, we have an isomorphism of modules
\begin{equation}
   HW^{k}(  V \boxtimes \Tn,  E ) = \Hom(V,  HW^{k}( \Tn,  E ) ).
\end{equation}
In the right hand side, the module structure is given by the action of $\End(V) $ on $V$, and the (possibly non-trivial) action of $HW^{0}( \Tn )$ on $ HW^{k}( \Tn,  E ) $.  Using the isomorphism between $HW^{0}( \Tn )$  and the group ring of $\pi_{1}(N)$, we may think of this action as giving a (possibly non-trivial) local system over $N $  with fibre $ HW^{k}( \Tn,  E )  $.  Let $E_{N}^{k}$ denote this local system, and observe that we now have an identification of modules
\begin{equation}
     HW^{k}(  V \boxtimes \Tn,  E ) = HW^{*}(V \boxtimes \Tn, E_{N}^{k}).
\end{equation}

The condition that $E$ be supported in finitely many degrees implies that only finitely many of these cohomology groups do not vanish.  In particular, as a module over $  HW^{*}(  V \boxtimes \Tn) $, the image of $E $ can be expressed as an iterated cone of the modules associated to $\{ E_{N}^{k} \}_{k=-\infty}^{+\infty} $.  Since both $E $ and $ E_{N}^{k} $ lie in the category split-generated by $ V \boxtimes \Tn $, we conclude that $ E $ is an iterated cone of the local systems $ \{ E_{N}^{k} \}_{k=-\infty}^{+\infty} $.  As these local systems are supported on $N$, we have proved the desired result.
\end{proof}

We now restrict to the case $N$ is simply connected, which implies that  $HW^{0}(\Tn)$ has rank one and all local systems on $N$ are trivial:
\begin{lem} \label{lem:local_system_iso_system_on_base}
If $N$ is simply connected, and $E$ is a local system of vector spaces on a closed Lagrangian then
 \begin{equation}
   HW^{*}(\Tn,E)
 \end{equation}
is supported in a single degree, and, up to shift, $E$ is isomorphic, in $\Seidel(\TN)$ to a trivial local system over the zero section. 
\end{lem}
\begin{proof}
From the previous Lemma, we know that $E$ may be expressed as a twisted complex on the trivial local systems with fibres $  HW^{k}( \Tn,  E )  $.  Lemma \ref{lem:Floer_infinity-structure-cpatible-with-tensor} implies that such local systems are quasi-isomorphic  to (possibly infinite) direct sums of the zero section.

Since the self-Floer cohomology of $N$ is supported in non-negative degrees, we are in the situation of Lemma \ref{lem:co-connective_result}.  Namely, we have an algebra $S=HW^{*}(N)$, and a twisted complex built from the vector spaces $ HW^{k}( \Tn,  E )  $.  Since $E$ is supported on a compact Lagrangian, $ HW^{*}(E,E) $ is supported in non-negative degrees, and hence so is the endomorphism algebra of this twisted complex.  We conclude from Lemma \ref{lem:co-connective_result}  that the twisted complex is supported in one degree, hence that $E$ is isomorphic to a local system of trivial vector spaces on $N$.
\end{proof}

With this in mind, we shall end this section with the proof of Theorem \ref{thm:simply_connected} in this special case:
\begin{proof}[Proof of Theorem \ref{thm:simply_connected} in the simply connected case]
A result of Fukaya, Seidel, and Smith (see Appendix \ref{sec:FSS}) implies that any closed exact Lagrangian $Q$ is  quasi-isomorphic   to the zero section in $\Seidel(\TN)$.  Lemma \ref{lem:Floer_infinity-structure-cpatible-with-tensor}, implies that trivial local systems over $N$ are also isomorphic, in $\Seidel(\TN)  $, to trivial local systems over $Q$ of the same rank.  From Lemma \ref{lem:local_system_iso_system_on_base}, we conclude that every local system over $Q$ defines an object of $ \Seidel(\TN)  $   which is isomorphic to a trivial local system, up to shift.

In order to conclude that every local system on $Q$ is indeed trivial, we appeal to the results of Appendix \ref{sec:floer-morse-class} (in particular, Lemma \ref{lem:floer_class}), which show that a quasi-isomorphism in $ \Seidel(\TN)  $ implies a  quasi-isomorphism in the (classical) category of local systems of chain complexes over $Q$.  In this category, the morphism spaces are given by
\begin{equation}
  H^{*}(Q, \End(E^1,E^2))
\end{equation}
and the degree $0$ part of this graded group is generated by  global maps of local systems.  In particular, an isomorphism in this category is the same as an isomorphism of local systems in the usual sense.
\end{proof}

\subsection{A generalisation of Lemma \ref{lem:E-iso-complex-local-systems}}
\label{sec:gener-lemma-ref}

While it shall not be used it in this paper, we record a generalisation of Lemma \ref{lem:E-iso-complex-local-systems}, which should clarify the feature of cotangent bundles that is being used:  Let $M$ be a Liouville manifold, and $L$ an exact Lagrangian which resolves the diagonal.  We assume that
\begin{equation} \label{eq:assumption_1_beyond_cotangent}
  \parbox{35em}{the wrapped Floer cohomology $HW^{*}(L)$ is supported in non-positive degrees.}
\end{equation}
From this we would like to extract a criterion for a finite collection of compact Lagrangians generating the subcategory of $\Seidel(M)$ whose objects are bounded complexes supported on compact Lagrangians.  Observe that Lemmata  \ref{lem:filtration_degree} and \ref{lem:determined_by_vector_space} apply to the algebra $ HW^{*}(L) $, so the argument of Lemma \ref{lem:E-iso-complex-local-systems} shows that every complex supported on a compact Lagrangian, and whose underlying local systems are non-zero in only finitely many degrees, defines a module over $HW^{*}(L)  $  which lies in the category generated by the modules induced by  ring homomorphisms
\begin{equation}
  HW^{0}(L) \to \End(U)
\end{equation}
for a vector space $U$.  Let us say that such a module is indecomposable if it does not factor through $\End(U_1) \oplus \End(U_2)  $ for a non-trivial decomposition of $U$.

Let us now consider a collection of closed exact Lagrangians $\sQ$ with the following property:
\begin{equation}\label{eq:assumption_2_beyond_cotangent}
    \parbox{35em}{each indecomposable module of $ HW^{0}(L) $ is isomorphic to $HW^{*}(L,F)$ for some local system $F$ supported on a Lagrangian lying in $\sQ$.}
\end{equation}
In particular, as modules over $HW^{*}(L)  $ any local system of bounded complexes over a closed manifold lies in the category generated by local systems over Lagrangians in $\sQ$. 

The main point of this paper is that the assignment of  $E \mapsto HW^{*}(L,E)  $ may not define a fully faithful embedding from the category $\Seidel(M)$ to the category of modules over $HW^{*}(L)$.  Rather, if we fix the dimension of $E$, there is a vector space $V$ such that the assignment $E \mapsto HW^{*}(V \boxtimes L,E)  $ defines a fully faithful embedding into  the category of modules over $HW^{*}(V \boxtimes L)$.  But for compact Lagrangians, this module is simply $\Hom(V, HW^{*}(L,E) )$, so any local system $E$ supported on a closed Lagrangian lies in the category of modules over $HW^{*}(V \boxtimes L)  $ which is generated by local systems over Lagrangians in $\sQ$.  Assuming $V$ to be large enough for Theorem \ref{thm:localising} to hold, we conclude:
\begin{prop}
  Under assumptions \eqref{eq:assumption_1_beyond_cotangent} and \eqref{eq:assumption_2_beyond_cotangent}, the category of local systems supported on Lagrangians in $\sQ$ generates the subcategory of $\Seidel(M)$ whose objects are local systems of bounded complexes, supported on closed exact Lagrangians.
\end{prop}

\section{Bimodules and twisted complexes} \label{sec:bimod-twist-compl}
In this section, we shall reduce Theorem \ref{thm:localising} to two technical results which we shall state after introducing the relevant ingredients.  Throughout, we are assuming that we have a pair of objects in $\Seidel(M)$, one of which is a trivial local system on a Lagrangian $L$, and the other an arbitrary local system $E$.

The main technical ingredient is a bimodule $ \Bimod_{L}(E)  $ over $CW^{*}(L)$ which we shall construct in Section \ref{sec':bimod}.  Whenever $E$ is a local system of finite rank supported over a compact Lagrangian, this bimodule is simply the tensor product of the left and right modules associated to $E$:
\begin{equation} \label{eq:basic_finite_dim_morphisms}
  CW^{*}(L,E) \otimes CW^{*}(E,L).
\end{equation}
In general, however, this tensor product is not ``large enough'', and $ \Bimod_{L}(E)  $ is the appropriate replacement whenever the morphism spaces are not finite dimensional (see Equation \eqref{eq:bimodule_L_E}).

In Section \ref{sec:bimod-homom} we construct a degree $n$ map of bimodules
\begin{equation}
  \Delta \co CW^{*}(L)  \to \Bimod_{L}(E) 
\end{equation}
whose source is the diagonal bimodule, which is obtained by counting holomorphic discs with two outgoing boundary punctures. Such a map induces a morphism on Hochschild homology groups which we denote
\begin{equation}
 HH_{*}(\Delta) \co HH_{*-n}( CW^{*}(L))  \to HH_{*}\left( CW^{*}(L), \Bimod_{L}(E) \right).
\end{equation}
In Section \ref{sec:an-evaluation-map}, we shall show the existence of a natural evaluation map
\begin{equation} \label{eq:cohomology_map_induced_mu}
  H^*(\mu) \co HH_{*}\left( CW^{*}(L), \Bimod_{L}(E) \right) \to HW^{*}(E)
\end{equation}
 defined using only the curves counted in the $A_{\infty}$ structure of $\Seidel(M)$.  Whenever the bimodule reduces to the basic case of Equation \eqref{eq:bimodule_L_E}, this evaluation map generalises the product 
 \begin{equation}
   CW^{*}(L,E) \otimes CW^{*}(E,L) \to CW^{*}(E).
 \end{equation}

By counting discs with one interior and one boundary puncture, we shall also construct a map
\begin{equation}
H^{*}( \CO) \co SH^*(M) \to HW^{*}(E)
\end{equation}
in Section \ref{sec:from-sympl-cohom} which takes the identity in symplectic cohomology to the identity of $E$.  We can now state the main technical result of the paper which is proved in Section \ref{sec:homotopy}:
\begin{prop} \label{prop:commutative_diagram}
These maps fit in a commutative diagram
\begin{equation} \label{eq:diagram_bimodules}
\xymatrix{ HH_{*-n}( CW^{*}(L)) \ar[r]^<<<<<{HH_{*}(\Delta)} \ar[d]^{H^*(\OC)}  & HH_{*}\left( CW^{*}(L),  \Bimod_{L}(E) \right) \ar[d]^{H_{*}(\mu)}   \\  
SH^{*}(M) \ar[r]^{H^*(\CO)}  & HW^{*}(E).   }
\end{equation}
\end{prop}

In order to conclude Theorem \ref{thm:localising}, we must have a different interpretation of $HH_{*}\left( CW^{*}(L),    \Bimod_{L}(E) \right)$.  We shall therefore show in Section \ref{sec:twist-compl} that there exists a directed system of  twisted complexes $\Univ_{L}^{n}(E)$ in the category of modules over $\Seidel(M) $, which are built from an action filtration on the trivial local systems on $L$ with fibre $CW^{*}(L,E)$, and which carry a canonical map
  \begin{equation}
    \tau \co  \colim_{n} \Univ_{L}^{n}(E) \to \cY_{E},
  \end{equation}
where $\cY_{E}$ is the image of $E$ under the Yoneda embedding.  Note that, even though we work at the chain level, the colimit on the left hand side is of a highly benign nature:  the maps in the directed system are in fact inclusions of twisted complexes.

We shall prove the following result in  Section \ref{sec:comp-two-eval}:
\begin{lem} \label{lem:evaluation_is_composition}
There exists a commutative diagram
\begin{equation} \label{eq:bimodule_to_evaluation}
\xymatrix{ HH_{*} \left( CW^{*}(L), \Bimod_{L}(E) \right) \ar[r] \ar[d]^{H_{*}(\mu)}  &  \colim_{n}   \Hom_{\mod-\Seidel(M)  }^{*}(\cY_{E},  \Univ_{L}^{n}(E)) \ar[d]^{\mu^{2}(\tau, \_)}   \\  
HW^{*}(E) \ar[r]^{\cong}  &  \End_{\mod-\Seidel(M)  }(\cY_{E})  }
\end{equation}
\end{lem}

With this at hand, we can prove our extension of the generation criterion:
\begin{proof}[Proof of Theorem \ref{thm:localising}]
The commutativity of Diagram \eqref{eq:diagram_bimodules} and the hypothesis of Theorem \ref{thm:localising}  imply that the identity of $HW^{*}(E)$ lies in the image of $H_{*}(\mu)$.  Using the commutativity of Diagram \eqref{eq:bimodule_to_evaluation}, we conclude that there exists some $n$ such that we have a morphism in
  \begin{equation}  \Hom_{\mod-\Seidel(M) }^{*}(\cY_{E},  \Univ_{L}^{n}(E))   \end{equation} 
whose product with $\tau$ is the identity of $\cY_{E}$.  We conclude that $\cY_{E}$ is a summand of $\Univ_{L}^{n}(E)  $, which is by construction an iterated cone of direct sums of the images of local systems on $L$ under the Yoneda embedding.

The statement about dimensions follows from the fact that the local systems used to build $ \Univ_{L}^{n}(E) $ arise from an action filtration on $CW^{*}(L,E)$.  If $E$ is finite, then every such local system is isomorphic to a finite direct sum of copies of $L$, so we conclude that $E$ lies in the category split generated by $L$.  Otherwise, the dimension of $ CW^{*}(L,E) $ is bounded by the dimension of $E$, so that $E$ indeed lies in the category split-generated by local systems on $L$ of dimension no greater than that of $E$.
\end{proof}

\subsection{Construction of the bimodule}  \label{sec':bimod}
Whenever $L$ resolves the diagonal and $E \in \Seidel(M)$ is supported on a closed Lagrangian $K$, we shall show that $E$ is a direct summand of the trivial local system of complexes $CW^{*}(L,E) $ over $L$, which we shall continue denoting $ CW^{*}(L,E) \boxtimes L $.   The general case will require using filtrations by finite dimensional subcomplexes, as we discuss at the end of this section.

We define
\begin{equation} \label{eq:bimodule_L_E}
\Bimod_{L}(E) \equiv  CW^{*}(E,  CW^{*}(L,E) \boxtimes L ) \textrm{ if the support of $U$ is compact},
\end{equation}
with the differential $\mu^1$ on the right hand side now denoted $\mu^{0|1|0}$.  Note that $L$ appears once as the input and once as the output of $CW^{*}(\_,\_)$ in the above equation, which accounts for the two different actions of $ CW^*(L) $.   More explicitly, the left module action comes from the natural inclusion
\begin{align}
\kappa \co CW^{*}(L) & \to CW^{*}(CW^{*}(L,E) \boxtimes L)    \cong \End( CW^{*}(L,E) )  \otimes CW^{*}(L) \\
x & \mapsto \id \otimes x,
\end{align}
which is a strict map of $A_{\infty}$ algebras in the sense that
\begin{equation}
\id \otimes  \mu^{d}(x^d, \ldots, x^1)  = \mu^{d}(  \id \otimes x^d, \ldots,  \id \otimes x^1).
\end{equation}
Whenever $d > 0$, the left module structure maps are given by
\begin{align} \label{eq:left_action}
  \mu^{d|1|0}  \co CW^{*}(L)^{\otimes d-1} \otimes CW^{*}(E,  CW^{*}(L,E) \boxtimes L ) & \to CW^{*}(E,  CW^{*}(L,E) \boxtimes L ) \\
(x^{d}, \ldots, x^{1}, \phi) & \mapsto \mu^{d+1}(   \id \otimes x^d, \ldots,  \id \otimes x^1, \phi).
\end{align}

The right module maps are given by the action of $CW^{*}(L)$ on the vector space $ CW^{*}(L,E) $.  To lift this to a map on $CW^{*}(E,  CW^{*}(L,E) \boxtimes L )  $, recall that this is, by definition, a direct sum
\begin{equation} \label{eq:direct_sum_decomp_tensor_with_L}
  \bigoplus_{x \in \Chord(K,L)} \Hom( E_{x(0)}, CW^{*}(L,E) ).
\end{equation}
The right module action for $d >0$ preserves this decomposition into direct summands, and is given on each such summand as follows
\begin{align} \label{eq:pre-right_action}
  \mu^{0|1|d}  \co  \Hom( E_{x(0)}, CW^{*}(L,E) )   \otimes CW^{*}(L)^{\otimes d-1}  & \to  \Hom( E_{x(0)}, CW^{*}(L,E) )  \\
 \mu^{0|1|d}(\phi, x^{d}, \ldots, x^{1}) & = \mu^{d+1}( \_ , x^d, \ldots, x^1) \circ \phi
\end{align}

We claim that the left and right actions commute, which follows from the following general fact: any endomorphism $\Psi$ of $ CW^{*}(L,E)$  induces an endomorphism of $ CW^{*}(E,  CW^{*}(L,E) \boxtimes L ) $ which acts on each summand 
\begin{equation}
  \Hom( E_{x(0)}, CW^{*}(L,E) )
\end{equation}
by post composition.  Since the parallel transport maps on $ CW^{*}(L,E)  $ are trivial,  composition with $\Psi$ commutes with the module action of $ CW^{*}(L) $ on $ CW^{*}(E,  CW^{*}(L,E) \boxtimes L )  $:
\begin{equation}
   \mu^{d+1}(   \id \otimes x^d, \ldots,  \id \otimes x^1, \phi) \circ \Psi =   \mu^{d+1}(   \id \otimes x^d, \ldots,  \id \otimes x^1, \phi \circ \Psi).
\end{equation}
Since the right module structure is defined using such linear maps $\Psi$, we conclude:
\begin{lem}
  The complex $\Bimod_{L}(E)$, equipped with structure maps $\mu^{r|1|s}$ which  vanish if both $r$ and $s$ are strictly positive, and are otherwise given by Equations \eqref{eq:left_action} and \eqref{eq:pre-right_action} is a bimodule over $CW^{*}(L)$.
\end{lem}

We shall find the above description of the right module structure a bit inconvenient in the discussion that follows, so we shall assume, for simplicity, that $L$ admits a strict unit $\id_{L}$, i.e. that $\mu^{d}(x^{d}, \ldots, \id_{L}, \ldots, x^1)$ vanishes unless $d=2$ in which case it is the identity.  In this case, we obtain a map
\begin{align}
\iota \co  CW^{*}(L)^{\otimes d} & \to CW^{*}( CW^{*}(L,E) \boxtimes L  )  \cong \End( CW^{*}(L,E) )  \otimes CW^{*}(L)  \\
(x^{d}, \ldots, x^{1}) & \mapsto \mu^{d+1}(\_, x^{d}, \ldots, x^{1}) \otimes \id_{L}.
\end{align}
The right module action defined above can be written instead as
\begin{align} \label{eq:right_action}
    \mu^{0|1|d}  \co   CW^{*}(E,  CW^{*}(L,E) \boxtimes L ) \otimes CW^{*}(L)^{\otimes d} & \to CW^{*}(E,  CW^{*}(L,E) \boxtimes L ) \\
(\phi , x^{d}, \ldots, x^{1}) & \mapsto \mu^{2}(   \mu^{d+1}(\_, x^{d}, \ldots, x^{1}) \otimes \id_{L} , \phi).
\end{align}

\subsubsection{Non-closed Lagrangians and action filtrations}
We now consider the case of a general Lagrangian:
\begin{defin} \label{def:bimodule_non-closed}
  If $E$ is supported on $K$, we define $\Bimod_{L}(E)$  to be the double direct sum
  \begin{equation}
    \label{eq:bimodule_explicit}
    \Bimod_{L}(E) = \bigoplus_{\substack{y \in \Chord(K,L) \\ x \in \Chord(L,K)} } \Hom(E_{y(0)}, E_{x(1)}) 
  \end{equation} \qed
\end{defin}
With this description,  one may easily show that the counts of pseudo-holomorphic curves used in the case of closed Lagrangians define operations $\mu^{d|1|0}$ and $\mu^{0|1|d}$ which make $  \Bimod_{L}(E) $  into an $A_{\infty}$ bimodule.  We shall give an equivalent definition by observing that
\begin{equation} \label{eq:univ_bimodule}
\Bimod_{L}(E) \equiv \colim_{a} CW^{*}(E,  CW^{*}_{\geq a}(L,E) \boxtimes L ).
\end{equation}
The right hand side is the directed colimit with respect to the inclusions  of subcomplexes $ CW^{*}_{\geq a}(L,E)   \to  CW^{*}_{\geq b}(L,E)  $ whenever $b < a$.   Lemma \ref{lem:properness_chords} implies that $ CW^{*}_{\geq a}(L,E)   $  vanishes whenever $a$ is sufficiently large, so we obtain a directed system of objects in $\Seidel(M)$
\begin{equation}
 0 \to  CW^{*}_{\geq N}(L,E) \boxtimes L \to  CW^{*}_{\geq N -1}(L,E) \boxtimes L \to \cdots \to CW^{*}_{\geq 0}(L,E) \boxtimes L \to \cdots
\end{equation}
The colimit of this system may not be representable in $\Seidel(M)$, but we may always form such a colimit in the category of modules, and shall use a similar idea to construct a bimodule structure. In order to keep our descriptions relatively simple, we shall assume that $L$ is strictly unital.  By going back to the definitions of morphism spaces in $\Seidel(M)$, the reader can check, as in the closed case, that the construction makes sense in full generality.  

To define the left module action, we use the natural inclusion 
\begin{align}
\iota_{a} \co CW^{*}(L) & \to CW^{*}(CW^{*}_{\geq a}(L,E) \boxtimes L)    \cong \End( CW^{*}_{\geq a}(L,E) )  \otimes CW^{*}(L) \\
x & \mapsto \id \otimes x,
\end{align}
Since the maps in the colimit are induced by inclusions of subcomplexes, the restriction of $\iota_{b} $ to $ CW^{*}_{\geq a}(L,E) \boxtimes L $ agrees with $\iota_a$, so the structure maps
\begin{align}
  \mu^{d|1|0}  \co CW^{*}(L)^{\otimes d-1} \otimes CW^{*}(E,  CW^{*}_{\geq a}(L,E) \boxtimes L ) & \to CW^{*}(E,  CW^{*}_{\geq a}(L,E) \boxtimes L ) \\
(x^{d}, \ldots, x^{1}, \phi) & \mapsto \mu^{d+1}( \iota_{a}(x^d), \ldots,  \iota_{a}(x^1), \phi)
\end{align}
define a left module action
\begin{equation}
    \mu^{d|1|0}  \co CW^{*}(L)^{\otimes d} \otimes \Bimod_{L}(E)   \to \Bimod_{L}(E) . \end{equation}

To define the right action, the action filtration allows us to write any element $ x^{d}\otimes  \cdots \otimes x^{1}  \in  CW^{*}(L)^{\otimes d}$ as lying in $ CW^{*}_{\geq c}(L)^{\otimes d} $  for some $c$.  Moreover, Lemma \ref{lem:action_filtration_infinity} implies that every product $\mu^{d+1}$ of such elements factors through the subspace supported by chords of sufficiently large action:
\begin{equation}
   \mu^{d+1} \co CW^{*}_{\geq a}(L,E) \otimes   CW^{*}_{\geq c}(L)^{\otimes d} \to  CW^{*}_{\geq b}(L,E) \textrm{ for some $b$.}
\end{equation}
This operation induces a map of trivial local systems over $L$
\begin{equation}
  CW^{*}_{\geq c}(L)^{\otimes d} \to CW^{*}\left( CW^{*}_{\geq a}(L,E) \boxtimes L,   CW^{*}_{\geq b}(L,E) \boxtimes L\right).
\end{equation}
For each given action level, we may therefore consider the composite
\begin{equation}
\xymatrix{    CW^{*}( E, CW^{*}_{\geq a}(L,E) \boxtimes L)  \otimes  CW^{*}_{\geq c}(L)^{\otimes d} \ar[d]   \\
CW^{*} ( CW^{*}_{\geq a}(L,E) \boxtimes L,  CW^{*}_{\geq b}(L,E) \boxtimes L )  \otimes  CW^{*}( E, CW^{*}_{\geq a}(L,E) \boxtimes L) \ar[d] \\
 CW^{*}( E, CW^{*}_{\geq b}(L,E) \boxtimes L) . }
\end{equation}
\begin{rem}
Note that we switch factors in the first map above.  This is an artefact of having constructed the right module structure by appealing to the existence of a unit on $L$, as can be seen by comparing Equation \eqref{eq:pre-right_action} with Equation \eqref{eq:right_action}, in which the appearance of $\mu_2$ is indicating the switch of factors.
\end{rem}

By construction, these maps are compatible with the inclusion of the subcomplexes $  CW^{*}_{\geq c}(L)^{\otimes d}  \subset   CW^{*}(L)^{\otimes d}   $  and $  CW^{*}_{\geq a}(L,E) \subset   CW^{*}(L,E)  $, so they define a map
\begin{equation}
  \mu^{0|1|d} \co   \Bimod_{L}(E) \otimes CW^{*}(L)^{\otimes d}  \to \Bimod_{L}(E).
\end{equation}

\subsection{An evaluation map on the cyclic bar complex} \label{sec:an-evaluation-map}
If we compute morphisms from $  CW^{*}(L,E) \boxtimes L$ to $E$, we find that
\begin{align}
  CW^{*}(   CW^{*}(L,E) \boxtimes L , E)  & \cong \bigoplus_{x \in \Chord(L,K)}  \Hom(  CW^{*}(L,E), E_{x(1)}) \\
& \cong \bigoplus_{x \in \Chord(L,K)}  \Hom\left(  \bigoplus_{y \in \Chord(L,K)} E_{y(1)} , E_{x(1)}\right).
\end{align}
Using the projection map $ \pi_{x} \co   \bigoplus_{y \in \Chord(L,K)} E_{y(1)} \to  E_{x(1)} $  which vanishes whenever $y \neq x$ and is the identity otherwise, we obtain  a canonical evaluation map
\begin{equation} \label{eq:evaluation_closed}
\tau_{0}  = \sum_{x \in \Chord(L,K)}  \pi_{x} \in CW^{*}(   CW^{*}(L,E) \boxtimes L , E) 
\end{equation}
which is a sum over finitely many chords whenever  $E$ is supported on a closed Lagrangian.   In this case, composition with $\tau_{0}$ defines an evaluation map
\begin{equation}
  CW^{*}(E,   CW^{*}(L,E) \boxtimes L)  \to CW^{*}(E,E)
\end{equation}
and the left hand side is $  \Bimod_{L}(E) $.  Without restriction on the support of $E$, we shall extend this to a map  whose source is the cyclic bar complex 
\begin{equation} \label{eq:direct_sum_CC_*_bimod}
 CC_{*}( CW^{*}(L), \Bimod_{L}(E)) \equiv \bigoplus_{d \geq 0} \Bimod_{L}(E) \otimes CW^{*}(L)^{\otimes d}
\end{equation}
which is equipped with the differential
\begin{multline} \label{eq:HH-differential}
  \psi \otimes x^d \otimes \ldots \otimes x^1 \mapsto \sum \mu^{k-1|1|0}(x^{k-1}, \ldots, x^{1}, \psi)  \otimes  x^{d} \otimes \cdots \otimes  x^{k} \\
+  \sum  \mu^{0|1|k-1} (\psi ,  x^{d},  \ldots,   x^{d-k+2}) \otimes x^{d-k+1} \otimes  \cdots \otimes x^{1} \\
+  \sum  \psi  \otimes x^{d} \otimes \cdots \otimes x^{\ell+k+1} \otimes \mu^{k}(x^{\ell+k},   \ldots,   x^{\ell+1}) \otimes x^{\ell} \otimes  \cdots \otimes x^{1}.
\end{multline}
The source of the map $\tau_0$ sits as a subcomplex in $CC_{*}( CW^{*}(L), \Bimod_{L}(E))  $ whenever $E$ is supported on a closed Lagrangian; in terms of the direct sum decomposition appearing in the right hand side of Equation  \eqref{eq:direct_sum_CC_*_bimod}, this is the summand $d=0$.

We shall now give an explicit construction of the map
\begin{equation}
  \mu \co CC_{*}( CW^{*}(L), \Bimod_{L}(E)) \to CW^{*}(E,E).
\end{equation}
Recall that $ \Bimod_{L}(E) $  is introduced in Definition \ref{def:bimodule_non-closed} as a direct sum
\begin{equation} \label{eq:double_direct_sum_cpt}
 \bigoplus_{\substack{ y \in \Chord(K,L) \\ x^{d+1} \in \Chord(L,K) } } \Hom \left(  E_{y(0)} ,   E_{x^{d+1}(1)}  \right).
\end{equation}
We write $\psi_{y}^{ x^{d+1}} $ for an element of one of these summands, and consider a sequence $\vx = (x^1, \ldots,x^d)$ of chords with endpoints on $L$.  We shall define 
\begin{equation}
  \mu(x^{d}, \ldots,  x^1, \psi_{y}^{ x^{d+1}} ) \in CW^{*}(E,E) =  \bigoplus_{x^0 \in \Chord(K,K)} \Hom \left(  E_{x^0(0)} ,   E_{x^{0}(1)}  \right)
\end{equation}
by counting holomorphic discs.

Each rigid disc $u \in \Disc(x^0, x^{d+1}, \vx, y) $ will contribute a homomorphism from $ E_{x^0(0)}  $ to $ E_{x^{0}(1)}  $.  As in the definition of the $A_{\infty}$ structure, we have parallel transport maps
\begin{align}
   \gamma_{u}^{0} \co E_{x^0(0)} & \to E^{0}_{y(0)} \\ 
  \gamma_{u}^{d+1} \co E_{x^{d+1}(1)} & \to E_{x^0(1)} .
\end{align}
Given an element $a $ of the fibre of $E$ at the starting point of $x^0$ we obtain an element in the fibre at the endpoint of $x^0$ by parallel transporting to the starting point of $y$, applying $\psi$, then transporting from the endpoint of $x^{d+1}$ to the endpoint of $x^0$:
\begin{equation}
  \label{eq:composition_bimodule}
    \mu(x^{d}, \ldots,  x^1, \psi_{y}^{ x^{d+1}} )(a) = \sum_{u}   \gamma_{u}^{d+1} \circ  \psi_{y}^{ x^{d+1}}  \circ  \gamma_{u}^{0}(a).
\end{equation}
This map is denoted $\mu$ because it is defined using the same moduli spaces as the higher products.  Indeed, note that we have an inclusion
\begin{equation}
  CW^{*}(E, L)  \otimes    CW^{*}(L,E)  \to   \Bimod_{L}(E)
\end{equation}
obtained by considering morphisms from $E_{y(0)}  $  to $  E_{x^{d+1}(1)}  $ whose image factor through a finite dimensional vector space.  In the special case where 
\begin{equation}
   \psi_{y}^{ x^{d+1}}  = \sum_{i} \phi^{i} \otimes \phi_{i},
\end{equation}
our construction is given by
\begin{equation}
   \mu(x^{d}, \ldots,  x^1, \psi_{y}^{ x^{d+1}} ) = \sum_{i} \mu^{d+2}(  \phi_{i}, x^{d}, \ldots,  x^1,   \phi^{i}).
\end{equation}
In particular, the $A_{\infty}$ equations imply that the restriction of $\mu$ to $CW^{*}(E, L)  \otimes    CW^{*}(L,E)   $  is a chain map.  

Even if we do not start with elements $ \Bimod_{L}(E)  $ of this special nature, it is still true that moduli spaces controlling $\mu^{d+2}$ and $\mu$ are the same.  In particular, by analysing the boundary of the compacitification of the  moduli spaces $ \Disc(x^0, x^{d+1}, \vx, y)  $, we find that there are two different types of strata:  If a rigid strip breaks off at the outgoing vertex $x^{0}$, we obtain the composition of $\mu$ with the differential on the target.  The remaining breakings all involve at least one input, and correspond to the composition of $\mu$ with the differential on the source.  We conclude:
\begin{lem} \label{lem:chain-map}
$\mu$ is a chain map. \qed
\end{lem}

\subsection{An infinite twisted complex}
We begin by writing $\cY_{F}$  for the right Yoneda module associated to an object $F$ of $\Seidel(M)$:
\begin{equation}
  \cY_{F}(E') = CW^{*}(E',F).
\end{equation}
Recall that an endomorphism of such a module consists of maps $t=(t^1, t^2, \ldots)$ where
\begin{equation} \label{eq:morphism_modules}
  t^{s+1} \in  \prod_{E'_i \in \Seidel(M)}\Hom\Big( \cY_{F}(E'_{s}) \otimes  CW^{*}(E'_{s-1}, E'_{s}) \otimes \cdots \otimes CW^{*}(E'_0, E'_1) , \cY_{F}(E'_0) \Big) 
\end{equation}
with the differential induced from the $A_{\infty}$ structure on $\Seidel(M)$:
\begin{align} \label{eq:differential_modules}
  (\partial t)^{s+1}( \underline{y} ,  y^{s},  \ldots ,  y^{1}) & = \sum t^{s-\ell+1}(\mu^{\ell+1}( \underline{y} ,  y^{s}, \ldots, y^{s - \ell +1}), y^{s - \ell}, \ldots , y^{1}) \\ \notag
& \quad +  \sum \mu^{s-\ell+1}(t^{\ell+1}( \underline{y} ,  y^{s}, \ldots, y^{s- \ell+1}), y^{s - \ell}, \ldots , y^{1})  \\ \notag
& \quad + \sum t^{s-\ell+1}( \underline{y} ,  y^{s}, \ldots, \mu^{\ell+1}(y^{k}, \ldots, y^{k - \ell }), y^{k - \ell-1}, \ldots , y^{1}). 
\end{align}

 We say that $t^1$ is the linear term of such an endomorphism, and that $t$ is \emph{linear} if all other terms vanish.  By construction, every element $x^r \otimes \cdots \otimes x^1$ of $CW^{*}(F)^{\otimes r}$ defines an endomorphism of $  \cY_{F}  $ given by the formula
\begin{equation}
  \label{eq:formula_action_bar_complex}
  \underline{y} \otimes y^{s} \otimes \cdots \otimes y^{1} \mapsto \mu^{r+s+1}( x^r ,  \ldots , x^1 , \underline{y} ,  y^{s},  \ldots ,  y^{1}).
\end{equation}

As a shorthand, we shall write $\cY_{L,E}$ for the Yoneda module associated to $ CW^{*}(L,E) \boxtimes L $.    The next result is the appropriate extension of the construction of Section \ref{sec':bimod} in the setting of Yoneda modules:
\begin{lem}
The action of $ CW^{*}(L)^{d}  $ on $ CW^{*}(L,E)  $ defines a homomorphism
\begin{equation}
  \label{eq:linear_homomorphism}
  \iota \co  CW^{*}(L)^{d} \to  \End_{\mod- \Seidel(M)}( \cY_{L,E} )
\end{equation}
whose image consists of linear endomorphisms.   Moreover, the inclusion of $CW^{*}(L)^{d}  $  into $ CW^{*}(  CW^{*}(L,E) \boxtimes  L)^{d} $ together with Equation \eqref{eq:formula_action_bar_complex} induces a different homomorphism
\begin{equation}
  \label{eq:non-linear_homomorphism}
  \kappa \co  CW^{*}(L)^{d} \to  \End_{\mod- \Seidel(M)}( \cY_{L,E} ).
\end{equation}\qed
\end{lem}

We shall use these two actions to define a natural differential on the direct sum
\begin{equation}
  \Univ_{L}(E) = \bigoplus_{d \geq 0} CW^{*}(L)^{\otimes d} \otimes \cY_{L,E}
\end{equation}
which gives an infinite twisted complex in the category of modules.   Recall that a twisted complex in a differential graded category consists of  the datum of such a direct sum, together with a strictly upper triangular matrix of morphisms
\begin{multline*}
  \delta_{d,d-k+1} \in \Hom \left(  CW^{*}(L)^{\otimes d} \otimes \cY_{L,E},   CW^{*}(L)^{\otimes d-k+1} \otimes \cY_{L,E} \right) \\ 
=  \Hom \left(  CW^{*}(L)^{\otimes d}, CW^{*}(L)^{\otimes d-k+1}   \otimes   \End_{\mod-  \Seidel(M)}( \cY_{L,E} )    \right)
\end{multline*}
such that
\begin{equation}
  \partial( \delta_{d,d-k+1}  ) = \sum_{ j  }\delta_{d,d-j+1} \cdot \delta_{d-j+1, d-k+1}.
\end{equation}

Given a generator $x^d \otimes \cdots \otimes x^1$ in the source of $\delta_{d,d-k+1} $, we explicitly write out this differential as the sum of three terms. The first is obtained by applying the higher products to consecutive factors
\begin{equation}
  \label{eq:first_differential_twisted}
  \sum_{\ell \leq d-k} x^{d} \otimes \cdots \otimes x^{\ell+k+1} \otimes \mu^{k}(x^{\ell+k},   \ldots,   x^{\ell+1}) \otimes x^{\ell} \otimes  \cdots \otimes x^{1}  \otimes \id_{  \cY_{L,E} }.
\end{equation}
The other two are obtained by applying $\kappa$ to terms at the beginning and $\iota$ to terms at the end
\begin{equation}  \label{eq:second_differential_twisted}
  x^{d} \otimes \cdots \otimes  x^{k}  \otimes \kappa(x^{k-1}, \ldots, x^{1}) +    x^{d-k+1} \otimes  \cdots \otimes x^{1}\otimes \iota(  x^{d},  \ldots,   x^{d-k+2}) . 
\end{equation}

The reader may easily check that the maps $\kappa$ and $\iota$ used to construct this twisted complex are compatible with the similarly named maps used in Section \ref{sec':bimod}.  By comparing the above three expressions with the formula for the differential in the cyclic bar complex given in Equation \eqref{eq:HH-differential}, we conclude:
\begin{lem}
The sum of Equations \eqref{eq:first_differential_twisted} and \eqref{eq:second_differential_twisted} equips $\Univ_{L}(E)   $ with the structure of a twisted complex.  \qed
\end{lem}

The Yoneda Lemma for $A_{\infty}$ categories (see, e.g. Section (1l) of \cite{seidel-book}) implies that the Yoneda homomorphism
\begin{equation}
CW^{*}(E, CW^{*}(L,E) \boxtimes L ) =  \cY_{L,E}(E) \to  \Hom_{\mod-  \Seidel(M)  } (\cY_{E},   \cY_{L,E})
\end{equation}
which acts on the right hand side by composition with an element of the left hand side is a quasi-isomorphism.  For each pair of integers $k \leq d$, we also have a map 
\begin{equation} \label{eq:lower_order_terms_map}
 CW^{*}(L)^{\otimes d} \otimes  CW^{*}(E, CW^{*}(L,E) \boxtimes L )  \to \Hom_{\mod-  \Seidel(M)  } (\cY_{E},  CW^{*}(L)^{\otimes d-k} \otimes \cY_{L,E}).
\end{equation}
The linear part consists of maps 
\begin{multline}
   CW^{*}(L)^{\otimes d} \otimes  CW^{*}(E, CW^{*}(L,E) \boxtimes L ) \otimes CW^{*}(E',E) \\
\to CW^{*}(L)^{\otimes d-k} \otimes CW^{*}(E',CW^{*}(L,E) \boxtimes L )
\end{multline}
for every object $E'$ in $\Seidel(M)$, and generalises Equation \eqref{eq:left_action} for the left action on $\Bimod_{L}(E)$
\begin{equation}
 x^{d} \otimes \cdots \otimes x^{1} \otimes \psi \otimes \underline{y}  \mapsto x^{d} \otimes \cdots \otimes x^{k+1} \otimes \mu^{k+2}(\id \otimes x^{k}, \ldots, \id \otimes x^1, \psi, \underline{y} ).
\end{equation}
The higher order terms are given by similar formulae:
\begin{multline}
 x^{d} \otimes \cdots \otimes x^{1} \otimes \psi \otimes \underline{y}  \otimes y^{s} \otimes \cdots \otimes y^{1} \mapsto \\ x^{d} \otimes \cdots \otimes x^{k+1} \otimes \mu^{k+2+s}(\id \otimes x^{k}, \ldots, \id \otimes x^1, \psi, \underline{y},y^{s}, \ldots , y^{1}  ).
\end{multline}

By taking the sum of all such maps, we obtain a chain map
\begin{equation} \label{eq:cyclic_complex_qi_hom_modules}
CC_{*}(CW^{*}(L), \Bimod_{L}(E))  \to \Hom_{\mod-  \Seidel(M)}(   \cY_{E},  \Univ_{L}(E)).
\end{equation} 

The next step is to interpret the evaluation map $\mu$ in terms of a composition in the category of twisted complexes.  To do this, we start by assuming that $E$ is supported on a closed Lagrangian.  First, we note that the analogue of the map $  \tau_{0} \in   CW^{*}(CW^{*}(L,E) \boxtimes L, E) $ is a natural homomorphism $  \cY_{L,E} \to \cY_{E} $ whose linear term, for every object $E'$ supported on a Lagrangian $L_{E'}$, is a map
\begin{equation} \label{eq:tautological_evaluation_bimodule}
 CW^{*}(E', CW^{*}(L,E) \boxtimes L)  \to CW^{*}(E', E) 
\end{equation}
obtained by extending the homomorphism
\begin{equation}
  \mu^{2} \co CW^{*}(L,E) \otimes CW^{*}(E',  L)  \to CW^{*}(E', E).
\end{equation}
 By ``extending'', we mean that we have a natural inclusion
\begin{equation}
\xymatrix{  CW^{*}(L,E) \otimes CW^{*}(E',  L) \ar[d] \ar[r]^<<<<<{=} & \displaystyle{\bigoplus_{x \in \Chord(L_{E'}, L)}} \Hom( E'_{x(0)}, \bF_{2}) \otimes CW^{*}(L,E) \ar[d] \\
CW^{*}(E', CW^{*}(L,E) \boxtimes L)  \ar[r]^<<<<<<{=} &\displaystyle{ \bigoplus_{x \in \Chord(L_{E'}, L)}} \Hom( E'_{x(0)}, CW^{*}(L,E) ) }
\end{equation}
and that the count of the curves which defines $\mu^2$ also defines a map in \eqref{eq:tautological_evaluation_bimodule}.  More precisely, Lemma \ref{lem:action_filtration_infinity} implies that whenever chords $x \in \Chord(L_{E'}, L)$ and $y \in \Chord(L,L_{E})$ are fixed, the moduli spaces $\Disc(z;x,y)$ are empty for all but finitely many chords $z \in \Chord(L_{E'}, L)$.  The count of such curves therefore defines a map
\begin{equation} \label{eq:product_map}
  \Hom( E'_{x(0)}, E_{y(0)}) \to CW^{*}(E',E).
\end{equation}
The key point here is that there are only finitely many chords  $y \in \Chord(L,L_{E})$, so that we have an isomorphism
\begin{equation}
  \bigoplus_{x \in \Chord(L_{E'}, L)}  \bigoplus_{y \in \Chord(L, L_{E})}    \Hom( E'_{x(0)}, E_{y(0)})  = CW^{*}(E', CW^{*}(L,E) \boxtimes L)
\end{equation}
which we use to define the map in Equation \eqref{eq:tautological_evaluation_bimodule} as the sum of the maps in Equation \eqref{eq:product_map}.

Using the higher products we similarly define an  evaluation map
\begin{equation}
CW^{*}(L)^{\otimes d} \otimes \cY_{L,E} \to \cY_{E}.
\end{equation}
Again, the linear term
\begin{align}
 CW^{*}(L) ^{\otimes d} \otimes  CW^{*}(E, CW^{*}(L,E) \boxtimes L)  \to CW^{*}(E, E) 
\end{align}
is obtained by extending the operation
\begin{equation}
\mu^{d+2} \co   CW^{*}(L,E) \otimes CW^{*}(L) ^{\otimes d} \otimes  CW^{*}(E, L)  \to CW^{*}(E, E) .
\end{equation}
The fact that the operations $\mu^d$ define an $A_{\infty}$ structure readily yields:
\begin{lem}
If $E$ is supported on a compact Lagrangian, these evaluation maps fit together into a closed morphism
\begin{equation}
\tau \co  \Univ_{L}(E) \to \cY_{E}
\end{equation} \qed
\end{lem}

\subsection{Construction of the twisted complexes} \label{sec:twist-compl}
Note that the construction of $ \Univ_{L}(E)  $ did not make any assumption on the Lagrangian supporting the local system $E$, but that the construction of $\tau$ relied on such an assumption because it required the existence of only finitely many chords.  For general $E$, $\tau$ may not be defined on $\Univ_{L}(E)   $, but we shall construct a replacement which admits such a map by considering $CW^{*}_{\geq b}(E,L) $,  the subcomplex of $ CW^{*}(E,L) $ generated by chords whose action is greater than $b$.

At the same time, even for a closed Lagrangian, the complex $  \Univ_{L}(E)  $ is infinite.  We shall give a construction which resolves both of these problems at once.  First, we choose real numbers $s_{d,n}$ defining an increasing and exhaustive filtration 
\begin{equation}
 \bigoplus_{d=0}^{n} CW^{*}(L)^{\otimes d}_{\geq s_{d,n}}  \subset  \bigoplus_{d =0}^{\infty} CW^{*}(L)^{\otimes d} 
\end{equation}
by finite dimensional subcomplexes (with respect to the bar differential) generated by the tensor product of morphisms supported on chords the sum of whose actions is bounded by $ s_{d,n}  $.  The condition that this is an exhaustive filtration is simply that for each $d$,
\begin{equation}
  \lim_{n \to +\infty} s_{d,n} = -\infty.
\end{equation}
To ensure that we have a filtration by subcomplexes, note that the bar differential is the sum of the contributions of the ordinary differential $\mu^{1}$ and of the higher product $\mu^{d}$.  We start, for each $n$, by choosing an arbitrary value for $s_{n,n}$, and note that the fact that $\mu^{1}$ strictly preserves the action filtration (see Lemma \ref{lem:differential_respects_filtration}) implies that the corresponding subspace is closed under the differential.  We then choose $s_{n-1,d}$ to (i) be smaller than $s_{n-1,n-1}$  and (ii) contain the image of $ CW^{*}(L)^{\otimes n}_{\geq s_{n,n}} $ under applying $\mu^{2}$ to two consecutive elements.  The second condition may be achieved because of Lemma \ref{lem:action_filtration_infinity}.  Proceeding by induction on both $n$ and $d$, we obtain the desired exhaustive filtration.

In a similar manner, we choose real numbers $r_{d,n}$ such that
\begin{equation}
   \bigoplus_{d=0}^{n} CW^{*}(L)^{\otimes d}_{\geq s_{d,n}} \otimes CW^{*}_{\geq r_{d,n}}(E,L)  \\ \subset \bigoplus_{d =0}^{\infty} CW^{*}(L)^{\otimes d}  \otimes CW^{*}(E,L) 
\end{equation}
is also an exhaustive filtration by finite dimensional subcomplexes.

If we let $ \cY^{\geq  r_{d,n}}_{L,E} $ denote the right Yoneda module corresponding to the local system $  CW^{*}_{\geq  r_{d,n}}(E,L) $ over $L$, then we define
\begin{equation}
 \Univ_{L}^{n}(E)  \equiv  \bigoplus_{d =0}^{n-1} CW^{*}(L)_{\geq  s_{d,n}}^{\otimes d} \otimes \cY^{r_{d,n}}_{L,E} \subset  \Univ_{L}(E).
 \end{equation}
Note that we are implicitly using the inclusion of Yoneda modules
\begin{equation}
  \cY^{\geq  r_{d,n}}_{L,E} \subset \cY_{L,E}
\end{equation}
associated to the inclusion of chain complexes.
\begin{lem}
The differential on $  \Univ_{L}(E) $ induces a differential on $  \Univ_{L}^{n}(E)  $.  Moreover, whenever $E$ is supported on a closed Lagrangian, we have
  \begin{equation}
     \Univ_{L}(E) = \colim_{n}  \Univ_{L}^{n}(E) 
  \end{equation}
\end{lem}
\begin{proof}
Since the differential on $   \Univ_{L}(E)  $ is defined using the differential on the bar complex, the fact that $ \Univ_{L}^{n}(E)  $ is preserved by the differential follows from our choices of action bounds.  Whenever $E$ is supported on a closed Lagrangian, $ CW^{*}_{\geq s_{d,n}}(E,L) $ stabilises for $n$ sufficiently large, implying the second result.
\end{proof}

\begin{rem} \label{rem:ranks_agree}
We will show that $E$ is a summand of $ \Univ_{L}^{n}(E)   $  for $n$ sufficiently large.  If $E$ is a finite-dimensional local system,  observe that all the objects appearing in the description of $ \Univ_{L}^{n}(E)   $ as a twisted complex are in fact quasi-isomorphic to finite direct sums of the image of $L$ under the Yoneda embedding.  The image of $E$ under this embedding therefore lies in the category generated by the image of $L$, so this implies that $E$ lies in the category split-generated by $L$ as objects of $\Seidel(M)$.

If $E$ is not finite dimensional, then $CW^{*}(E,L)$  has dimension equal to the dimension of $E$,  and hence $  \Univ_{L}^{n}(E)  $ is in fact constructed from the image, under the Yoneda embedding, of local systems over $L$ whose dimension, while infinite, is nonetheless equal to that of $E $.  In this case, we shall conclude that $E$ lies in the category split-generated by local systems on $L$ of dimension equal to the dimension of $E$.
\end{rem}

Now, for any object $E$, the construction of the previous section produces evaluation maps
\begin{equation}
\tau \co  \Univ_{L}^{n}(E) \to \cY_{E}
\end{equation}
which are compatible with the maps in the direct system.  In particular, we obtain an evaluation map
\begin{equation}
\tau \co  \colim_{n} \Univ^{n}_{L}(E) \to \cY_{E}.
\end{equation} 
Moreover, comparing the definitions of $ \Univ^{n}_{L}(E) $  and of $\Bimod_{L}(E)$, we note that the argument used to construct the chain map of Equation \eqref{eq:cyclic_complex_qi_hom_modules} extends to the non-closed case and gives a map
\begin{equation}
  \Bimod_{L}(E) \to  \colim_{n} \Univ^{n}_{L}(E) .
\end{equation}

\subsection{Comparing the two evaluation maps} \label{sec:comp-two-eval}

Starting with an element of the cyclic bar complex of $\Bimod_{L}(E)$, we may now apply Equation \eqref{eq:cyclic_complex_qi_hom_modules}, and then compose with $\tau$ to obtain an endomorphism of $\cY_{E}$, or we might map to $CW^{*}(E)$ using $\mu$ then apply the Yoneda embedding.   We now explain why,  at the level of cohomology, these two compositions agree:
\begin{proof}[Proof of Lemma \ref{lem:evaluation_is_composition}]
For simplicity, we shall only prove that Diagram \eqref{eq:bimodule_to_evaluation} commutes when restricted to the subcomplex
\begin{equation}
   \bigoplus_{d}   CW^{*}(L,E)  \otimes CW^{*}(L)^{\otimes d} \otimes  CW^{*}(E, L)  \subset \Bimod_{L}(E)
\end{equation}
where the element  $ \psi \in  CW^{*}(E, CW^{*}(L,E) \boxtimes L)  $  factors through a finite-dimensional local system over $L$.   This allows us to directly use the $A_{\infty}$ equation rather than return to the stratification of the moduli spaces of pseudo-holomorphic discs from which the $A_{\infty}$ equation arises.    

Consider such an element $ \sum_{i} \phi^i \otimes x^{d} \otimes \cdots \otimes x^{1}  \otimes \phi_i$ of $ \Bimod_{L}(E) $.  For any sequence of objects $\left(E'_s, \ldots E'_0 \right) $ of $\Seidel(M)$, we have two linear maps
\begin{align}
CW^{*}(E'_{s},E)  \otimes  CW^{*}(E'_{s-1}, E'_{s}) \otimes \cdots \otimes CW^{*}(E'_0, E'_1)  & \to CW^{*}(E'_{0},E) .
\end{align}
The first is obtained by applying $\mu$ then the Yoneda embedding, and assigns to $   \underline{y} \otimes y^{s} \otimes \cdots \otimes y^{1} $ 
\begin{equation}
\sum_{i} \mu^{s+2}( \mu^{d+2}( \phi^i,  x^{d}, \ldots ,  x^{1},  \phi_i  )     ,  \underline{y} ,  y^{s},  \ldots ,  y^{1}).  
\end{equation}

The second comes from applying the  map defined in Equation \eqref{eq:cyclic_complex_qi_hom_modules} then composing in the category of modules:
\begin{equation}
  \sum_{i,k,\ell}  \mu^{d-k+s-\ell+2}( \phi^i,  x^{d}, \ldots , x^{k+1},   \mu^{k+2+\ell}(x^{k}, \ldots,  x^{1},  \phi_i   ,   \underline{y} ,  y^{s},  \ldots ,  y^{s-\ell+1}) , y^{s-\ell}, \ldots, y^{1})
\end{equation}
The fact that the operations $\mu^{d}$ define an $A_{\infty}$ structure implies that a homotopy between these two maps is provided by 
\begin{equation} \label{eq:homotopy_evaluations}
\sum_{i} \mu^{d+s+3}(  \phi^i,  x^{d}, \ldots ,  x^{1},  \phi_i    ,   \underline{y} ,  y^{s},  \ldots ,  y^{1} ).
\end{equation}
In order to see this, note that the other terms in the $A_{\infty}$ equation are
\begin{align*}
&  \sum_{i,\ell} \mu^{d+s-\ell+3}(  \phi^i,  x^{d}, \ldots ,  x^{1},  \phi_i    , \mu^{\ell+1}\left(  \underline{y} ,  y^{s},  \ldots ,  y^{s-\ell+1}\right), y^{s-\ell}, \ldots, y^{1}  ) \\
 &  \sum_{i,\ell} \mu^{s-\ell+1} (   \mu^{d+\ell+3}(  \phi^i,  x^{d}, \ldots ,  x^{1},  \phi_i    ,  \underline{y} ,  y^{s},  \ldots ,  y^{s-\ell+1} ), y^{s-\ell}, \ldots, y^{1}  )  \\
 &  \sum_{i,\ell}   \mu^{d+s-\ell+3}(  \phi^i,  x^{d}, \ldots ,  x^{1},  \phi_i    ,  \underline{y} ,  y^{s},  \ldots ,  \mu^{\ell+1}(y^{k}, \ldots, y^{k - \ell }), \ldots, y^{1}  )  \\
&\sum_{i,k,\ell} \mu^{d-k+s+3}( \mu^{k+1}(   \phi^i,  x^{d}, \ldots ,  x^{d-k-1}), x^{d-k}, \ldots, x^{1},  \phi_i    ,  \underline{y} ,  y^{s},  \ldots ,  y^{1} )  \\
&\sum_{i,k,\ell}  \sum \mu^{d-k+s+3}(  \phi^i,  x^{d}, \ldots ,  x^{k+1}, \mu^{k+1}( x^{k}, \ldots, x^{1},  \phi_i  )  ,  \underline{y} ,  y^{s},  \ldots ,  y^{1} )  \\
&\sum_{i,k,\ell}  \sum \mu^{d-k+s+4}( \phi^i,  x^{d}, \ldots ,  x^{\ell+k+1}, \mu^{k}(x^{\ell+k},   \ldots,   x^{\ell+1}) ,  x^{\ell}, \ldots, x^{1},  \phi_i    ,  \underline{y} ,  y^{s},  \ldots ,  y^{1} ).
\end{align*}
By comparing with Equation \eqref{eq:morphism_modules}, we find that the first three terms correspond to applying Equation \eqref{eq:homotopy_evaluations}, then the differential in $\End_{\mod-\Seidel(M)}(\cY(E)) $.  On the other hand, the last three correspond to the applying the differential in the cyclic bar complex of $ \Bimod_{L}(E) $, then Equation \eqref{eq:homotopy_evaluations}.  We conclude that the two compositions  in Diagram \eqref{eq:bimodule_to_evaluation} commute.
\end{proof}

\section{Constructions of the structure maps}\label{sec:constr-struct-maps}
In this section, we prove Proposition \ref{prop:commutative_diagram} which requires constructing all the maps in the following diagram, as well as a homotopy between the two compositions:
\begin{equation}
\xymatrix{ CC_{*-n}( CW^{*}(L)) \ar[r]^<<<<<{CC_{*}(\Delta)} \ar[d]^{\OC}  & CC_{*}\left( CW^{*}(L),  \Bimod_{L}(E) \right) \ar[d]^{\mu}   \\  
SC^{*}(M) \ar[r]^{\CO}  & CW^{*}(E).   }
\end{equation}
We shall review the constructions of \cite{generate}, and adapt them to presence of local systems.  We start with the case that requires a modicum of thought.

\subsection{The bimodule homomorphism} \label{sec:bimod-homom}
The first map we shall construct is the bimodule homomorphism
\begin{equation}
  \Delta \co CW^{*}(L) \to \Bimod_{L}(E) .
\end{equation}
Recall that such a homomorphism consists of a collection of maps
\begin{equation}
    \Delta^{r|1|s} \co CW^{*}(L)^{\otimes r} \otimes CW^{*}(L)  \otimes CW^{*}(L)^{\otimes s}  \to \Bimod_{L}(E)
\end{equation}
for each non-negative integer $r$ and $s$ satisfying the appropriate $A_{\infty}$ equation.

For $r=s=0$, this map is defined by considering pseudo-holomorphic maps from a disc with one incoming end $\underline{\xi}$ and two outgoing ends $\xi^{0}$ and $\xi^{-1}$.  Given a chord $\underline{x}$ with both endpoints on $L$, and a pair of chords $x^{-1} \in \Chord(K,L)  $ and $ x^0 \in \Chord(L,K) $, we write 
\begin{equation}
  \Disc_{0|1|0}(x^{-1},x^0;\underline{x})
\end{equation}
for the corresponding moduli space of discs with moving Lagrangian boundary conditions (see Section 3.3 of \cite{generate}).  As in Lemma \ref{lem:diffeo_for_evaluation}, we can apply a diffeomorphism so that an element $u$ of this moduli space maps the components of the boundary to $K$ or $L$, as shown on Figure \ref{fig:comult}.
\begin{rem}
If the local system $E$ is trivial, our clockwise conventions for ordering points on the boundary would have the map defined by the curve in Figure \ref{fig:comult} take value in $CW^{*}(K,L) \otimes CW^{*}(L,K)  $.  With our algebraic conventions on composition, the natural evaluation map goes from this tensor product to $CW^{*}(L,L)$.  In order to obtain an evaluation map into $CW^{*}(K,K)$, we must switch the order of the outputs.  This is represented, in Figure \ref{fig:comult}, by braiding the ends corresponding to the outputs.
\end{rem}
\begin{figure}
  \centering
  \includegraphics{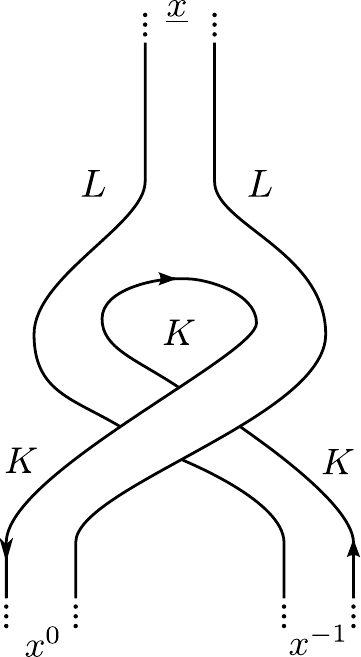}
  \caption{ }
  \label{fig:comult}
\end{figure}

In particular, we obtain a map 
\begin{equation}
  \gamma_{u} \co E_{x^{-1}(0)} \to  E_{x^{0}(1)}
\end{equation}
by parallel transport along the image under $u$ of the boundary component of the disc mapping to $K$.  We define
\begin{align} \label{eq:coproduct_map}
  \Delta^{0|1|0} \co CW^{*}(L) & \to \bigoplus_{\substack{x^{-1} \in  \Chord(K,L) \\   x^0 \in \Chord(L,K)} }  \Hom(E_{x^{-1}(0)} , E_{x^{0}(1)}) \\
\underline{x} & \mapsto \sum_{\substack{ u \in  \Disc_{0|1|0}(x^{-1},x^0;\underline{x}) \\ u \textrm{ is rigid.}}  } \gamma_{u}. \notag
\end{align}
\begin{lem}
The map $   \Delta^{0|1|0}  $  defines a degree $n$ chain homomorphism
\begin{equation}
  CW^{*}(L) \to  \Bimod_{L}(E)
\end{equation}
\end{lem}
\begin{proof}
Recall that we have computed that
\begin{equation}
  CW^{*}(E, CW^{*}(L,E) \boxtimes L) = \bigoplus_{x^{-1} \in  \Chord(K,L) }  \Hom\left(E_{x^{-1}(0)} ,  \bigoplus_{  x^0 \in \Chord(L,K)}   E_{x^{0}(1)}\right)
\end{equation}
and that the image of $  \Bimod_{L}(E) $ in this direct sum of $\Hom$ spaces is precisely the subspace where each morphism factors through a finite direct sum as in the right hand side of Equation \eqref{eq:coproduct_map}.  The fact that $   \Delta^{0|1|0} $  is a chain map then follows from the usual description of the Gromov compactification of $ \Disc_{0|1|0}(x^{-1},x^0;\underline{x}) $, which  is obtained by allowing breakings of strips at the three ends; these correspond to applying the differential in either $ CW^{*}(L) $ or  $ CW^{*}(L,E)  $  or $ CW^{*}(E, L) $. 
\end{proof}

The construction of the higher morphisms $  \Delta^{r|1|s} $ is now entirely straightforward.  Given sequences of chords with both endpoints on $L$ which we denote $\vx = \{ x^{k} \}_{k=1}^{r}$,  $\underline{x}$,  $\vx[|] = \{ x^{|k} \}_{k=1}^{s}$ as well as $x^{-1} \in \Chord(K, L) $ and  $x^0 \in \Chord(L, K) $, we define a moduli space
\begin{equation}
  \Disc_{r|1|s}(x^{-1}, x^0 ;  \vx[|], \underline{x}, \vx)
\end{equation}
of pseudo-holomorphic maps from a disc with $r+s+1$ incoming ends, and two outgoing ends (see Section 4.2 of \cite{generate}).  After applying the appropriate diffeomorphism, an element of this moduli space maps one of the boundary segments to $K$, converging in one direction to $ x^{-1}(0)$ and in the other to $ x^{0}(1)$.  Writing $\gamma_{u}$ for the associated parallel transport map, we define
\begin{equation}
  \Delta^{r|1|s}(x^r, \ldots, x^1, \underline{x}, x^{|1}, \ldots,   x^{|s}  ) = \sum_{\substack{ u \in  \Disc_{r|1|s}(x^{-1}, x^0 ;  \vx[|], \underline{x}, \vx)  \\ u \textrm{ is rigid.}}  } \gamma_{u}.
\end{equation}

The proof that the collection of maps $ \Delta^{r|1|s} $ satisfies the axioms of an $A_{\infty}$-bimodule map follows from analysing the Gromov compactification of $   \Disc_{r|1|s}(x^{-1}, x^0 ;  \vx[|], \underline{x}, \vx) $ as in Section 4.2 of \cite{generate}.
\subsection{From symplectic cohomology to wrapped Floer cohomology} \label{sec:from-sympl-cohom}
Next, we construct the map
\begin{equation}
  \CO \co SC^{*}(M) \to CW^{*}(E).
\end{equation}

Recall that $SC^{*}(M)$ is a chain complex generated by the time-$1$ \emph{periodic orbits} of a time-dependent Hamiltonian in $M$.  Given such an orbit $y$, and a time-$1$ chord $x$ with both endpoints on $K$, we consider the moduli space
\begin{equation}
  \Disc_{1}^{1}(x;y)
\end{equation}
of pseudo-holomorphic discs mapping the boundary to $K$, equipped with one interior puncture converging to $y$ and a negative end along the boundary converging to $x$, as shown in Figure ~\ref{fig:CO}.
\begin{figure}
  \centering
  \includegraphics{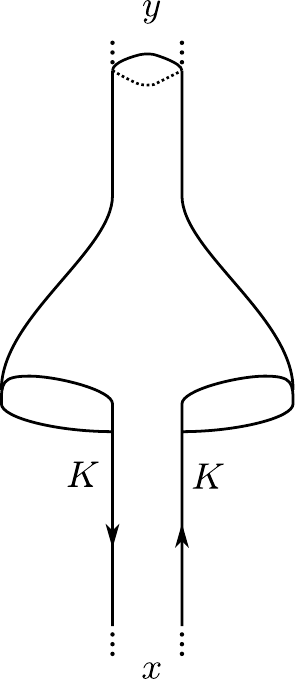}
  \caption{ }
  \label{fig:CO}
\end{figure}
Given such a punctured disc $u$, the parallel transport map along the boundary defines a homomorphism
\begin{equation}
    \gamma_{u} \co E_{x(0)} \to  E_{x(1)},
\end{equation}
and we define
\begin{equation}
  \CO(y) = \sum_{\substack{ u \in   \Disc_{1}^{1}(x;y) \\ u \textrm{ is rigid.}}  } \gamma_{u}.
\end{equation}
The boundary of those moduli spaces $  \Disc_{1}^{1}(x;y) $ of dimension $1$ has two types of strata.  Either (1) a pseudo-holomorphic cylinder  breaks at the interior puncture, or (2) a strip breaks at the boundary puncture.  The first corresponds to the differential in $SC^*(M)$, while the second corresponds to the differential in $ CW^{*}(E)$.

 Recall that the identity in $ SC^{*}(M) $  is obtained by counting rigid pseudo-holomorphic planes converging to a Hamiltonian orbit.  To see that its image under $\CO$ represents the identity, one glues such a plane to the interior puncture of the disc controlling the operation $\CO$ whose outgoing end is then glued to one of the inputs of the disc controlling the operation $\mu^2$ on $CW^{*}(E)  $.  The resulting disc now has one input and one output, and the count of such discs is homotopic to the identity map of  $CW^{*}(E)$. 
\begin{lem}
  The image under  $H^{*}(\CO)$ of the identity of $SH^{*}(M)$  is the identity of $HW^{*}(E)$. \qed
\end{lem}

\subsection{The homotopy} \label{sec:homotopy}
The map induced by $\Delta$ at the level of Hochschild chains is given by
\begin{multline}
  \label{eq:hochschild_map_induced_by_coproduct}
CC_{*}(\Delta) (x^{d} \otimes \ldots \otimes x^{1}) =  \\ \sum_{r+1+s \leq d}  \Delta^{r|1|s}(x^{r}, \ldots, x^{1}, \underline{x}^{d}, x^{d-1}, \ldots, x^{d-s}) \otimes x^{d-s-1} \otimes \cdots \otimes x^{r+1}.
\end{multline}
Note our convention of underlining the input of $\Delta^{r|1|s}$ which is an element of the bimodule.

Using the definition of $\mu$ given in Section \ref{sec:an-evaluation-map}, we conclude that the composition $ \mu \circ CC_{*}(\Delta) $  is given by counting pseudo-holomorphic curves consisting of two components (see the right most picture of Figure \ref{fig:broken} for one of the possibilities whenever $d=3$):
\begin{itemize}
\item An element of the moduli space $ \Disc_{r|1|s}(x^{-1}, x^0 ; x^{r}, \ldots, x^{1}, \underline{x}^{d}, x^{d-1}, \ldots, x^{d-s} )$ for some integers $r$ and $s$ whose sum is smaller than $d$ and for a pair of chords  $x^{-1} \in \Chord(K, L) $ and  $x^0 \in \Chord(L, K) $.
\item An element of the moduli space  $\Disc_{d-r-s+2}(x; x^{-1},  x^{r+1} , \ldots , x^{d-s-1} , x^{0})$ for some chord $x$ with both endpoints on $K$.
\end{itemize}
\begin{figure}
  \centering
  \includegraphics{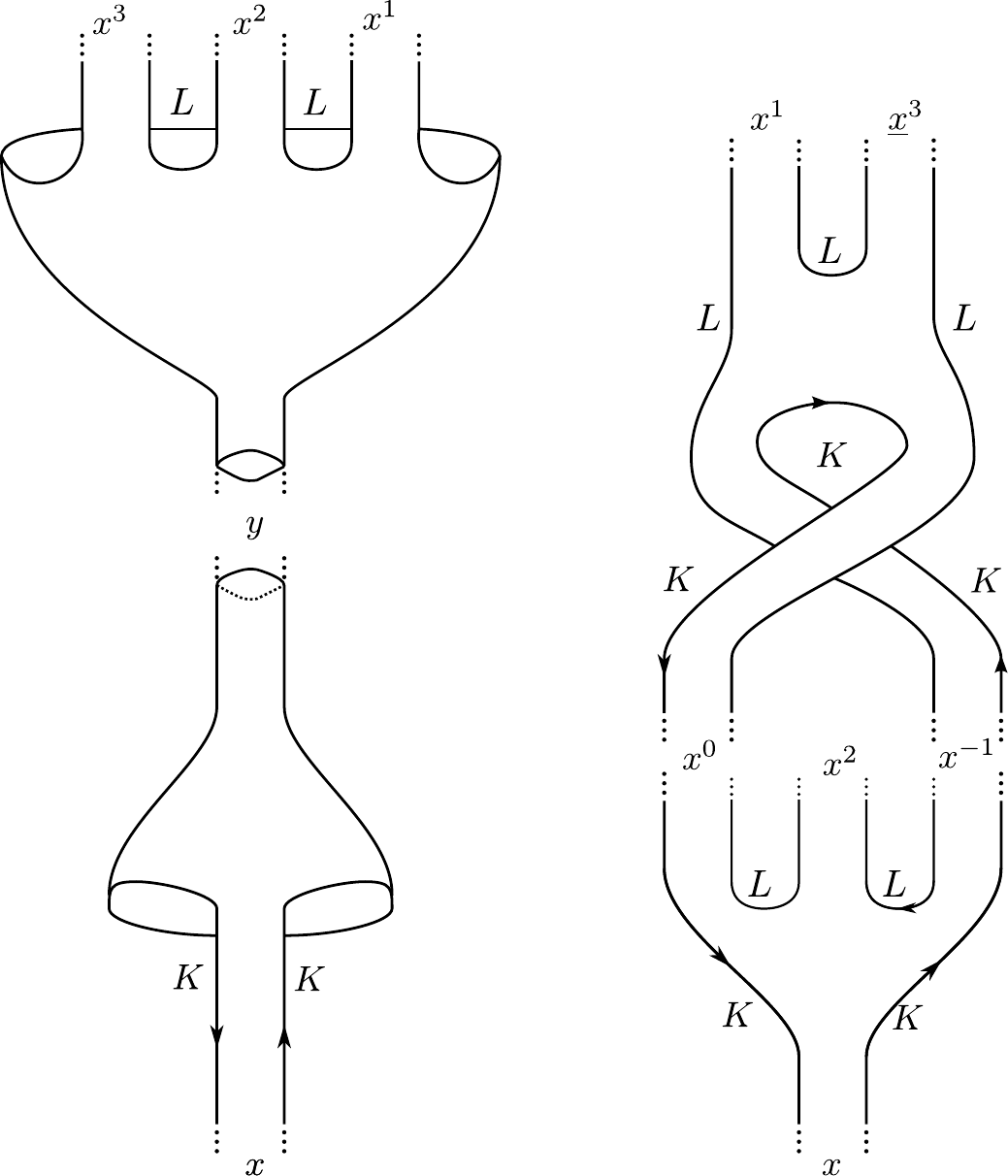}
  \caption{ }
  \label{fig:broken}
\end{figure}
The composition  $\CO \circ \OC$ is obtained by counting  pseudo holomorphic curves with an interior node as shown on the left side of Figure \ref{fig:broken}.  The two components of such a curve are:
\begin{itemize}
\item A pseudo-holomorphic disc with boundary mapping to $L$, $d$ positive boundary ends ordered counterclockwise with the $k$\th end converging to $x^k$, and one interior puncture converging to a periodic orbit $y$ (i.e. to a generator of $SC^{*}(M)$).  A detailed description of this moduli space is given in Section 5.3 of \cite{generate}.
\item An element of the moduli space $  \Disc_{1}^{1}(x;y) $ for some periodic chord $x$ with both endpoints on $K$.
\end{itemize}

By inspecting Figure \ref{fig:broken}, the reader can see that the result of gluing both nodes on the right side gives an annulus.  One of the observations of \cite{generate} is that the abstract curves which control the compositions $\CO \circ \OC  $ and $ \mu \circ CC_{*}(\Delta) $  arise as the boundary of the moduli spaces $\Ann^{-}_{d} $ of annuli with $d$ ordered  marked points on one boundary component which are labelled incoming, and $1$ outgoing marked point on the other satisfying the following property:  the annulus with $2$ marked points obtained by forgetting all the incoming marked points but the last admits a biholomorphism to a region in the plane bounded by two circles centered at the origin, which takes  one marked point to a positive real number, and the other to a negative real number.  If $d=1$, the abstract moduli space of such annuli has dimension $1$, and the ratio between the radii of the images of the boundary circles gives a parametrisation by the interval $(1,+\infty)$.  We compactify this moduli space by adding two singular configurations: whenever the ratio of radii converges to $1$ we add a singular surface with two nodes, obtained by gluing a pair of discs with $3$ boundary marked points, while when the ratio converges to $+\infty$ we add a pair of discs, each carrying both an interior and a boundary marked point, glued along the interior marked points.

 A version of the Cardy relation implies that the two types of broken curves we are considering form part of the boundary of a moduli space of maps from punctured annuli with $d$ incoming ends on one circle converging to the inputs $(x^{d}, \ldots, x^{1})$ and the connecting segments mapping to $L$, and a single outgoing end on the other boundary circle which converges to the output $x$ such that the complementary segment maps to $K$.  We write
\begin{equation}
   \Ann^{-}_{d}(x; x^{1}, \ldots, \underline{x}^{d} ) 
\end{equation}
for the moduli space of maps from such annuli to $M$. Note that we again have a map
\begin{equation}
   \gamma_{u} \co E_{x(0)} \to  E_{x(1)}
\end{equation}
associated to each element $u \in  \Ann^{-}_{d}(x; x^{1}, \ldots, \underline{x}^{d} ) $, obtained by parallel transport along the boundary component carrying the outgoing end.  In addition to the two broken curves representing the compositions  $\CO \circ \OC$ and $ \mu \circ CC_{*}(\Delta) $, the boundary of this moduli space is given by  taking the following products:
\begin{align}\label{eq:boundary_stratum_annuli_differential_CC_0}
& \Disc(x; x^0)  \times \Ann_{d}(x^0;  x^{1}, \ldots, \underline{x}^{d} ) \\
& \Ann_{d-r-s+1}^{-}(x;   x^{r+1} , \ldots , x^{d-s-1}, \underline{x}^0) \times  \Disc_{r+s+1}(\underline{x}^0; x^{r}, \ldots, x^{1}, \underline{x}^{d}, x^{d-1}, \ldots, x^{d-s} ) \label{eq:boundary_stratum_annuli_differential_CC_1} \\
&  \Ann_{d-k}^{-}(x;   x^{1} , \ldots , x^{\ell}, x^0, x^{\ell+k+1}, \ldots, \underline{x}^d) \times  \Disc_{k}(x^0; x^{\ell+1},  \ldots, x^{\ell+k} ) .
\label{eq:boundary_stratum_annuli_differential_CC}
\end{align}

\begin{proof}[Proof of Proposition \ref{prop:commutative_diagram}]
We claim that the map
\begin{align}
  CC_{*}( CW^{*}(L))  & \to  CW^{*}(E) \\ 
 x^d \otimes \ldots \otimes x^1 & \mapsto \sum_{\substack{u \in \Ann^{-}_{d}(x; x^{1}, \ldots, \underline{x}^{d})  \\ u \textrm{ is rigid}}} \gamma_{u} 
\end{align}
defines a homotopy between the two compositions.  This follows immediately from our description of the boundary of $ \Ann^{-}_{d}(x; x^{1}, \ldots, \underline{x}^{d}) $  since two of the types of boundary strata correspond to these compositions, while the other two correspond to applying the differential in $CW^{*}(E)$ to the image of the  homotopy as in Equation \eqref{eq:boundary_stratum_annuli_differential_CC_0} or applying the differential in the cyclic bar complex followed by the homotopy as in Equations \eqref{eq:boundary_stratum_annuli_differential_CC_1} and \eqref{eq:boundary_stratum_annuli_differential_CC}.
\end{proof}

\section{The Fukaya category of a cover} \label{sec:fukaya-categ-cover}
In this section, we shall develop the tools that allow us to prove Theorem \ref{thm:simply_connected} in the non-simply connected case.   For any Liouville manifold $M$,  we shall construct a wrapped Fukaya category $\tSeidel(M) $ associated to a covering space $\pi \co \tM \to M$, whose objects are covers of Lagrangians in $M$ equipped with local systems.  Given a Lagrangian in $M$ the full inverse image in $\tM$ is such a cover, and we shall prove the following result in Section \ref{sec:floer-cochains-cover}:
\begin{prop} \label{prop:pullback_functor_exists}
The assignment $L \to \pi^{-1}(L)$ extends to a \emph{pullback functor} 
\begin{equation} \label{eq:pullback}
  \pi^{*} \co \Seidel(M) \to \tSeidel(M)
\end{equation}
which takes a local system to its classical pullback.  
\end{prop}

Using this result to determine the homotopy type of closed Lagrangians requires knowing that the correspondence between Floer and classical cohomology applies in our setting.   The precise result we need is proved in Section \ref{sec:proof-lemma-refl}, using the adjunction between local systems on a space and it covers and the results of Appendix \ref{sec:floer-morse-class}.
\begin{lem} \label{lem:floer_class_cover}
If $ \tQ $  is a cover of a closed exact Lagrangian $Q \in M$, and $\tE^1$ and $ \tE^2$ are local systems on $\tQ$, then
\begin{equation} \label{eq:floer=classical-cover}
  HW^{*}(\tE^1,\tE^2) = H^{*}(\tQ, \Hom(\tE^1,\tE^2)).
\end{equation}
Moreover, the natural compositions on both sides agree.
\end{lem}
With this result at hand, we may extend the arguments presented in Section \ref{sec:simply-conn-cotang} to non-simply connected cotangent bundles:

\begin{proof}[Proof of Theorem  \ref{thm:simply_connected}]
Given a closed exact Lagrangian $Q$ in $\TN$, we shall prove that any cover  $\tQ \subset \TtN$ is connected and that the map on fundamental groups induced by the inclusion of $Q$ in $\TN$ is an isomorphism.

We first prove connectedness:  We know from the results of Fukaya, Seidel and Smith (see Appendix \ref{sec:FSS}) that $Q$ is equivalent as an object of the Fukaya category of $\TN$ to a local system of rank $1$ over the zero section of $\TN$.  From Proposition \ref{prop:pullback_functor_exists}, we conclude that $\pi^{-1}(Q)$ is equivalent, in the Fukaya category of $\TtN$ to a local system of rank $1$ over $\tN$.  Since $\tN $ is simply connected, we conclude that $\pi^{-1}(Q)$ and the zero section of $\TtN$ are  quasi-isomorphic objects of the category $\tSeidel(\TN) $, hence have isomorphic self-Floer cohomology.  Applying Equation \eqref{eq:floer=classical-cover} first to $\TtN$ then to $\pi^{-1}(Q)  $, we see that $ HW^{0}(\tN,\tN)  $ has rank $1$, hence $H^{0}( \pi^{-1}(Q), \bF_{2})  $ has rank $1$.  We conclude that the inverse image of $Q$ in $\Ttn$ is connected, which implies that the map $\pi_{1}(Q) \to \pi_{1}(\TN)$ is surjective.

To prove the simple connectivity of $\tQ$, we consider the local system $E_{ \pi_{1}(Q) } $ on $Q$ corresponding to the group ring $\bF_{2}[\pi_{1}(Q)]$.  By Lemma \ref{lem:E-iso-complex-local-systems}, this is quasi-isomorphic to an iterated extension of local systems on $N$.  Applying the pullback functor, we find that $\pi^{*}(  E_{ \pi_{1}(Q)})  $ is quasi-isomorphic to an iterated extension of local systems on $\tN$ which are necessarily trivial by simple connectivity.

As in the simply connected case, we are now again in the situation of Lemma \ref{lem:co-connective_result}.  Namely, Lemma \ref{lem:floer_class_cover} implies that $HW^{*}(\tN)  $  is isomorphic to the ordinary cohomology of $\tN$, hence is supported in non-negative degree.  Moreover, the endomorphism algebra of $ \pi^{*}(  E_{ \pi_{1}(Q)})  $ is itself supported in non-negative degree.  The argument given in Lemma \ref{lem:E-iso-complex-local-systems} implies that $ \pi^{*}(  E_{ \pi_{1}(Q)})  $ is in fact isomorphic to a trivial local system on $ \tN $, supported in a single degree.

Having established previously that $\tQ$ and the zero section of $\TtN$ are quasi-isomorphic objects of the category $\tSeidel(\TN) $, we see that $ \pi^{*}(  E_{ \pi_{1}(Q)})  $ is isomorphic, in $\tSeidel(\TN) $, to a trivial local system on $ \tQ $.  Using Lemma \ref{lem:floer_class_cover}, we conclude that $  \pi^{*}(  E_{ \pi_{1}(Q)})  $ is in fact trivial as a local system on $ \tQ $.

To complete the proof, note that the representation of $\pi_{1}( \tQ )   $ corresponding to $  \pi^{*}(  E_{ \pi_{1}(Q)})  $ is induced, from the group ring of  $\pi_{1}( Q )  $, by the inclusion   of $ \pi_{1}( \tQ )  $ as a subgroup of $\pi_{1}( Q )  $.  In particular, the group ring of $ \pi_{1}( \tQ )  $ sits as a subrepresentation, hence is also trivial.  This implies that $ \pi_{1}( \tQ )   $  vanishes, and therefore that the inclusion of $\pi_{1}( Q )  $ in $\pi_{1}(\TN)$ is an isomorphism.
\end{proof}
\begin{rem}
 By using the adjunction between local systems on $Q$ and $\tQ$, we have avoided proving an analogue of Theorem \ref{thm:localising} for the Fukaya category we attach to the cover, and in particular, showing that a Lagrangian which resolves the diagonal in $M$ has inverse image in $\tM$ which also resolves the diagonal.  There is a superficial problem in proving such a statement, coming from the fact that annuli are not simply connected and hence not all maps from an annulus to $M$ lift to $\tM$; however, this would not really affect such an argument because all annuli we need to consider, when proving the existence of a resolution of the diagonal,  lie in the trivial homotopy class.  More problematic is the fact that one should not expect that the Hochschild homology and cohomology of $\tSeidel(M) $ be isomorphic whenever the cover is not finite, and that symplectic cohomology can be defined in at least two different ways depending on support conditions. 
\end{rem}

\subsection{Floer cochains in a cover} \label{sec:floer-cochains-cover}
Let $\pi \co \tM \to M$ be a covering space.  For the purpose of this section, we consider a collection of exact Lagrangians  $L \subset M$ together with covers which embed in $\tM$:
\begin{equation}
  \xymatrix{ \tL      \ar[d] \ar@{^{(}->}[r] & \tM \ar[d] \\
L \ar@{^{(}->}[r] & M.}
\end{equation}
We shall construct a category $\tSeidel(M)$ whose objects are complexes of local systems over $\tL$, with $\tWrap(M)$ denoting the subcategory where the local systems are trivial of rank $1$.  We do not require that $\tL$ be the full inverse image of $L$ under the covering map $\pi$, nor do we require $\tL$ to be connected.   Whenever $\tM$ is a finite cover,  the Floer cohomology groups of $\tL$ with various other Lagrangians are the same in $\tWrap(M)$ as they would be in any reasonable definition of the wrapped Fukaya category of $\tM$.  However, some care is to be taken when $\tL$ is an infinite cover of $L$, as we explain in Example \ref{ex:all-fibres} below.

First, we define $\Chord(\tL^0, \tL^1)$ to be the set of lifts $\tx \co [0,1] \to \tM$ of chords $x \in \Chord(L^0,L^1)$ such that $\tx(0) \in \tL^0$    and $\tx(1) \in \tL^1$.  Given local systems $\tE^{0}  $ and $\tE^{1}$ over these Lagrangians, we define
\begin{equation}
  \label{eq:CW-cover}
  CW^{k}(\tE^0, \tE^1)  = \bigoplus_{\substack{x \in \Chord(L^0, L^1) \\ |x| = k}} \prod_{\substack{\tx \in  \Chord(\tL^0, \tL^1) \\ \pi(\tx) = x}} \Hom(  \tE^0_{\tx(0)} , \tE^1_{\tx(1)} ).
\end{equation}
In particular, whenever $\tE^0$ and $\tE^1$ are both trivial, this is the vector space generated by those infinite sequences of elements of $\Chord(\tL^0, \tL^1)  $ which involve only lifts of finitely many elements of $ \Chord(L^0,L^1) $. 

To define a differential, we denote by $\Disc(\tx^0,\tx^1)$ the set of lifts of elements of $\Disc(x^0,x^1) $ which converge to $\tx^0$ at one end and to $\tx^1$ at the other.  Given such a disc $\tu$, we have parallel transport maps
\begin{align}
  \gamma_{\tu}^{0} \co \tE^{0}_{\tx^0(0)} & \to \tE^{0}_{\tx^1(0)} \\ 
  \gamma_{\tu}^{1} \co \tE^{1}_{\tx^1(1)} & \to \tE^{1}_{\tx^0(1)}
\end{align}
as in Section \ref{sec:diff}.

\begin{lem} \label{lem:differential_cover}
The map
\begin{align} \notag
\mu^1 \co   CW^{*}(\tE^0, \tE^1)  & \mapsto CW^{*}(\tE^0, \tE^1)  \\
\sum  \phi_{\tx^1}  & \mapsto    \sum_{\substack{ \tu \in  \Disc(\tx^0,\tx^1)\\ u \textrm{ is rigid}} }  \gamma_{\tu}^{1} \circ \phi_{\tx^1}   \circ    \gamma_{\tu}^{0}
  \label{eq:differential-cover}
\end{align}
defines a differential.
\end{lem}
\begin{proof}
To keep matters clearer, we consider the case where $ \tE^0$ and  $  \tE^1 $ are both trivial.  We can then write a generator of the Floer complex as an infinite sum
\begin{equation}
  \sum  a_{\tx^1}  \tx^1 \in     CW^{*}(\tL^0, \tL^1)
\end{equation}
with $a_{\tx^1} \in \bF_{2}$.  First, we note that the vanishing of the coefficients away from the lifts of finitely many chords is preserved by $\mu^1$, since $ \Disc(x^0,x^1) $ is empty for all but finitely many $x^0$ whenever $x^1$ is fixed.  We must therefore prove that the expression
\begin{equation}
  \sum_{\tx^1}  a_{\tx^1} \sum | \Disc(\tx^0,\tx^1) |
\end{equation}
has only finitely many non-zero terms for each $\tx^0 \in  \Chord(\tL^0, \tL^1)   $.   Since $CW^{*}(\tL^0, \tL^1)  $ consists of a direct sum over $x \in \Chord(L^0, L^1)  $, it suffices to consider the case where $a_{\tx^1} $ vanishes unless $\pi(\tx^1)=x^1$ for a given $x^1$.   In this case, it suffices to prove the finiteness of
\begin{equation}
  \sum_{\pi(\tx^1)=x^1}  | \Disc(\tx^0,\tx^1) |
\end{equation}
for fixed $x^1$ and $\tx^0$.  This expression is strictly smaller than the number of elements of $  \Disc(x^0,x^1)  $ which is finite by Gromov compactness.   Knowing that the differential is well defined, the fact that it squares to zero follows from the usual techniques if we keep track of  relative homotopy classes of paths as in the proof of Lemma \ref{lem:differential_squares_0}.
\end{proof}

To understand the construction, it helps to keep in mind a few examples in the case where $M $ is a cotangent bundle $\TN$, and $\tM$ is the universal cover:
\begin{example} \label{ex:one-fibre}
Assume $L$ is a cotangent fibre $\Tn$, and that $\tL$ is the  cotangent fibre $\Ttn$ of $\tN$ at some fixed lift $\tn$ of $n$.  In this case, the condition that a lift of a chord in $T^*N$ have both endpoints on $\Ttn$ implies that $x$ is a contractible chord, and moreover fixes a unique lift.  Noting that the equivalence between the wrapped Floer complex and the chains of the based loop space respects  relative homotopy classes,  we conclude that we have a quasi-isomorphism
\begin{equation}
 CW^{*}(\Ttn)  \cong C_{-*}(\Omega_{\tn} \tN).
\end{equation}
\end{example}
\begin{example} \label{ex:all-fibres}
Assuming still that $L$ is a cotangent fibre, we may choose $ \tL $  to be the union over all cotangent fibres lying over $L$.  In this case,
\begin{equation}
CW^{*}( \cup_{\pi(\tn)=n} \Ttn  ) = \bigoplus_{g \in \pi_{1}(N)} \prod_{ \pi(\tn)=n }  CW^{*}(\Ttn, T^{*}_{g(\tn)} \tN). 
\end{equation}
Applying the equivalence with the homology of the based loop space, we find that this is quasi isomorphic to
\begin{equation} \label{eq:based_loops_all_lifts}
 \bigoplus_{g \in \pi_{1}(N)} \prod_{ \pi(\tn)=n }   C_{-*}(\Omega_{\tn, g(\tn)} \tN) .
\end{equation}
The interested reader may compare this group with the notion of finite propagation matrices which appears in coarse geometry, and  which we would obtain by passing to cohomology, and considering only the degree $0$ part of $ H_{-*}(\Omega_{\tn, g(\tn)} \tN) $ which is free of rank $1$.  By fixing a basepoint $\tn_{0}$ for  $\tN$, we may write every lift of $\tn$ as $h \tn_{0}$ for $h \in  \pi_{1}(N)$.  In particular, every element of
\begin{equation}
 \bigoplus_{g \in \pi_{1}(N)} \prod_{h \in \pi_{1}(N) } H_{0}(\Omega_{h \tn_{0}, g \cdot h (\tn_{0})} \tN)  
\end{equation}
gives rise to a matrix whose rows are labelled by $h$ and columns by $gh$.  The matrices arising from this construction are called finite propagation matrices (see, e.g. Remark 2.2 of \cite{roe}).
\end{example}

\subsection{Construction of the category}
Let us now consider a collection $L^0, \ldots, L^{d}$ of Lagrangians in $M$, together with lifts $\tL^{0}, \ldots, \tL^{d}$ to $\tM$ and local systems $\tE^0, \ldots, \tE^{d}$ on these lifts. We shall define maps
\begin{equation}
  \mu^{d} \co CW^{*}(\tE^{d-1} ,\tE^{d}) \otimes \cdots \otimes CW^{*}(\tE^{0} ,\tE^{1})  \to  CW^{*}(\tE^{0} ,\tE^{d}) 
\end{equation}
which encode the  $A_{\infty}$ structure on the category $\tSeidel(M)$.    First, we recall from Section \ref{sec:wrapp-floer-cohom} that, given a sequence of chords $\vx = (x^1, \ldots, x^d)$ with $x^{k} \in \Chord(L^{k-1}, L^{k})$ and a chord $x^0 \in  \Chord(L^{0}, L^{d})$, the moduli space $\Disc(x^0, \vx)$ is the space of pseudo-holomorphic maps $u \co S \to M$ whose source is an abstract disc $S \in \Disc_{d}$.  After applying the appropriate diffeomorphism, the maps take the components of the boundary to $L^0, \ldots, L^d$ and converge along the $k$\th end to $x^{k}$.

Because the domain $S$ is contractible, every element of $\Disc(x^0, \vx)  $ lifts uniquely to $\tM$ once the lift of any point is specified.  In particular, if we choose lifts $\tx^{k} \in \Chord(\tL^{k-1}, \tL^{k})$ and $\tx^0 \in  \Chord(\tL^{0}, \tL^{d}) $, we define
\begin{equation}
  \Disc(\tx^0, \vtx)
\end{equation}
so be the set of those lifts which converge to the images of $\tx^{k}$ along the $k$\th end; this is a subset of $\Disc(x^0, \vx)  $.  Given such a lift $\tu$, we obtain parallel transport maps
 \begin{align}
  \gamma_{\tu}^{0} \co \tE^{0}_{\tx^0(0)} & \to E^{0}_{\tx^1(0)} \\ 
  \gamma_{\tu}^{k} \co \tE^{k}_{\tx^k(1)} & \to E^{k}_{\tx^{k+1}(0)} \textrm{ if  $1 \leq k \leq d-1$}  \\ 
  \gamma_{\tu}^{d} \co \tE^{d}_{\tx^d(1)} & \to E^{d}_{\tx^0(1)} 
\end{align} 
as in Lemma \ref{lem:diffeo_for_evaluation}.  If we are given  morphisms
\begin{equation}
  \sum_{\tx^{k} \in  \Chord(\tL^{k-1}, \tL^{k}) }   \phi^{\tx^k}_{k} \in CW^{*}(\tE^{k-1}, \tE^{k})  \subset \prod_{ \tx^{k} \in  \Chord(\tL^{k-1}, \tL^{k}) } \Hom( \tE^{k}_{\tx^k(0)}  ,  \tE^{k}_{\tx^k(1)} )
\end{equation}
 for each integer $k$ between $1$ and $d$, we define their $d$\th higher product via the formula
\begin{equation} \label{eq:higher_product_structure}
  \mu^{d}\left(\sum  \phi^{\tx^d}_{d}  , \ldots,  \sum  \phi^{\tx^1}_{1}  \right) = \sum_{ \substack{u \in \Disc(\tx^0, \vtx) \\  u \textrm{ is rigid} } } \gamma_{\tu}^{d} \circ \ \phi^{\tx^d}_{d}   \circ \cdots \circ \gamma_{\tu}^{1} \circ  \phi^{\tx^1}_{1}  \circ \gamma_{\tu}^{0}.
\end{equation}

We omit the proof of the next result, which follows the same argument as that of Lemma \ref{lem:differential_cover}.   The main point is that, if we  fix the chords $x^{k}$  for $1 \leq k \leq d$, and choose a lift $\tx^{0}$,  then the expression
\begin{equation}
   \sum_{\pi(\tx^{k}) = x^{k}}   | \Disc(\tx^0, \vtx) | 
\end{equation}
is finite by Gromov compactness.  In particular, even though each of the series $ \sum \phi^{\tx^k}_{k}  $ has infinitely many terms, there are only finitely many terms in their image under $\mu^d$ lying in each component
\begin{equation}
  \Hom( \tE^{0}_{\tx^0(0)}  ,  \tE^{d}_{\tx^0(1)} )
\end{equation}
for any given lift $ \tx^0 $.   
\begin{lem}
The operations $\mu^{d}$ define an $A_{\infty}$ structure on $ \tSeidel(M) $. \qed
\end{lem}

Given an exact Lagrangian $L \subset M$, the inverse image $\pi^{-1}(L) \subset \tM$ is an object of $ \tWrap(M) $.    Moreover, the pullback of a local system $E$ on $L$ gives a local system $\pi^{*}(E)$ on $ \pi^{-1}(L) $.
\begin{prop} \label{prop:functor_well_defined}
The assignment $ E \mapsto \pi^{*}(E)$ extends to an $A_{\infty}$ functor
\begin{equation}
  \pi^{*} \co  \Seidel(M) \to  \tSeidel(M) .
\end{equation}
\end{prop}
\begin{proof}
Given $\phi_{x} \in \Hom(E^{0}_{x(0)}, E^{1}_{x(1)})$, we write $\phi_{\tx}$ for the corresponding homomorphism from the fibre of $ \tE^{0} $ at $ \tx(0) $ to the fibre of $ \tE^{1} $ at $ \tx(1) $. The action on morphism spaces is given by
\begin{equation}
  \label{eq:pull_back_map}
    \pi^{*}(\phi_{x}) = \sum_{\pi(\tx) = x}  \phi_{\tx}.
\end{equation}
In Equation \eqref{eq:CW-cover}, we precisely allowed expressions of this form  in the morphism spaces of $\tSeidel(M)  $.  The proof that $  \pi^{*} $ strictly commutes with all the higher products follows from the fact the $A_{\infty}$ structure on $ \tSeidel(M)  $ is obtained by using lifts of discs from $M$.
\end{proof}
It is helpful to keep in mind the two canonical examples whenever $M = \TN$.  First, we continue Example \ref{ex:all-fibres}
\begin{example}
If we choose a quadratic Hamiltonian on the cotangent bundle, then the unit of $CW^{*}(\Tn)$ is represented by the critical point $p=0$, and the other generators of the wrapped Floer complex correspond to non-constant geodesics.  The pullback map $\pi^{*}$ takes every such geodesic to the sum of all its possible lifts, and in particular, the image of the unit under $\pi^{*}$ is the sum  of the intersection points of all the Lagrangian $\Ttn$ with the zero section.   Applying the equivalence with chains of the based loop space, we obtain the sum, over all possible lifts $\tn$ of $n$, of the constants loops at $\tn$.  Concatenation with these loops is the identity.
\end{example}
\begin{example}
 If we choose a Hamiltonian which is sufficiently $C^2$ small near $N$, then under the equivalence with Morse theory, the identity in $CW^{*}(N)$ is represented by the maxima of a Morse function, and the associated cycle in homology which is Poincar\'e dual to the identity is obtained by taking the closure of the descending manifolds, which happens to be all of $N$.  The inverse image in $CW^{*}(\tN)  $ again corresponds to the sum over all maxima, and the descending manifolds now cover all of $\tilde{N}$.  This makes sense, because the fundamental class of $\tilde{N}  $ lives in locally finite homology.  Applying the same argument to more general classes, we find that our pullback map $\pi^{*}$ agrees with the usual pullback of cohomology classes from $N$ to its cover.
\end{example}

\subsection{Proof of Lemma \ref{lem:floer_class_cover}}
\label{sec:proof-lemma-refl}

Generalising the case of the zero section of a cotangent bundle, consider any closed exact Lagrangian $Q \subset M$, and use a Hamiltonian function which is $C^{2}$ small near $Q$ to define the Floer complex.  In this case, all chords with endpoints on $Q$ are constant, and map to critical points of the restriction of $H$ to $Q$, so the fibres of any local system on $Q$ at the starting and ending points of a chord are the same.  In particular, given local systems $\tE^{1}$ and $\tE^{2}$, we have an isomorphism
 \begin{equation} \label{eq:hom_global_sections}
   CW^*(\tE^1, \tE^{2}) \cong CW^*(Q, \Hom(\tE^1, \tE^{2})) =  \bigoplus_{x \in \Chord(Q)} \prod_{\substack{\tx \in  \Chord(\tQ) \\ \pi(\tx) = x}} \Hom(  \tE^1_{\tx(1)} , \tE^2_{\tx(1)} ),
 \end{equation}
where we abuse notation by writing, as before, $Q$ for the trivial local system of rank $1$ on $Q$.  Also, note that the assumption that all chords are constant implies that $x(0)=x(1)$, which we use in the above expression.

In order to express these complexes in terms of the base, let us recall  that  the pushforward of a local system $\tE$ on $\tQ$ by $\pi$ is defined to be
\begin{equation} \label{eq:push_forward_equation}
  \pi_{*}(\tE)_{x} = \prod_{\pi(\tx) = x} \tE_{\tx}
\end{equation}
which is to be interpreted as the set of sections of the restriction of $\tE$ to the inverse image of $x$.
Applying this formula to Equation \eqref{eq:hom_global_sections}, and observing that the differential on $ CW^*(\tE^1, \tE^{2})  $ was defined by lifting discs from $\TN$, we conclude that we have an isomorphism of chain complexes.  
 \begin{equation} \label{eq:express_hom_by_pushforward}
  CW^*(\tE^1, \tE^{2})  = CW^*\left(\tQ,\pi_{*}\left( \Hom(  \tE^1 , \tE^2 )\right)\right). 
 \end{equation}

In order to prove the last remaining part of Lemma \ref{lem:floer_class_cover}, we establish a similar formula in Floer cohomology.  For convenience of notation, we consider the slightly more general situation studied in Section \ref{sec:gener-prod} and Appendix \ref{sec:floer-morse-class}.  Namely, we start with a  map of local systems
\begin{equation}
  \tE^{1,2} \otimes \tE^{0,1} \to \tE^{0,2}
\end{equation}
over $\tQ$, yielding a map of local systems
\begin{equation}
  \pi_{*}(\tE^{1,2}) \otimes   \pi_{*}(\tE^{0,1}) \to   \pi_{*}(\tE^{0,2})
\end{equation}
over $Q$.  In terms of the formula \eqref{eq:push_forward_equation} for the pushforward, one simply multiplies sections of $ \tE^{1,2} $ and $\tE^{0,1}  $ pointwise over the fibre of $\pi$ at a point $x$. In particular, using the product defined in Section \ref{sec:gener-prod}, we obtain a map
\begin{equation} \label{eq:pushforward_product_Floer}
   CW^*(Q,  \pi_{*}(\tE^{1,2})   ) \otimes CW^*(Q,  \pi_{*}(\tE^{0,1})   ) \to CW^*(Q,  \pi_{*}(\tE^{0,2})).  
\end{equation}

\begin{lem} \label{lem:compare_Floer_product_structures}
If $\tE^{i,j} = \Hom(\tE^i,\tE^j)$, the isomorphism of Equation \eqref{eq:express_hom_by_pushforward} intertwines the product in Equation \eqref{eq:pushforward_product_Floer} with the product defined in Equation \eqref{eq:higher_product_structure}. 
\end{lem}
\begin{proof}
  We use Equation \eqref{eq:express_hom_by_pushforward} to compare these two products. Let us fix lifts $\tx$ and $\ty$ of chords $x$ and $y$ in $\Chord(Q)$.  Given a disc $u$ with boundary on $Q$ which converges at the inputs to $x$ and $y$ and at the output to a chord $z$, both constructions induce maps on
\begin{equation}
  \Hom(\tE^{1}_{\ty(0)}, \tE^{2}_{\ty(1)}) \otimes  \Hom(\tE^{0}_{\tx(0)}, \tE^{1}_{\tx(1)}).
\end{equation}
For the product $\mu^2$, this map is non-zero if and only if the boundary segment of the disc which converges to $ x(1) $ at one end and $y(0)$ at the other lifts to a path connecting $ \tx(1)  $ and $ \ty(0) $.  On the other hand, the  product defined in Section \ref{sec:gener-prod} admits a non-trivial contribution if the lifts of the paths from $x(0)$ to $z(0)$ and from $y(1)$ to $z(1)$ which start at $\tx(0)$ and $\ty(1)$ satisfy the property that they end at $\tz(0)$ and $\tz(1)$ for the same lift $\tz$.  Using a Hamiltonian whose only chords with endpoints on $Q$ are constant, we use the fact that the starting and ending points of each chord are equal to conclude that  the contributions are indeed the same.
\end{proof}

We end this Section by completing the proof of its titular Lemma:
\begin{proof}[Proof of  Lemma \ref{lem:floer_class_cover}]
 At the cohomology level, Lemma \ref{lem:floer_class} asserts that the (classical) cohomology of $Q $ with coefficients in $ \pi_{*}( \Hom(  \tE^0 , \tE^1 )) $ is isomorphic to its Floer cohomology with those coefficients.  By the classical adjunction formula, we know that the first is isomorphic to the classical cohomology of $\tQ$ with coefficients in $  \Hom(  \tE^0 , \tE^1 ) $.  Together with the isomorphism in Equation \eqref{eq:express_hom_by_pushforward}, this proves that
\begin{equation}
    HW^*(\tE^0, \tE^{1})  = H^{*}(\tQ,  \Hom(  \tE^0 , \tE^1 ) ).
\end{equation}

To see that the product structures agree, note that the map of local systems
\begin{equation}
\pi_{*}( \Hom(  \tE^1 , \tE^2 ))   \otimes  \pi_{*}( \Hom(  \tE^0 , \tE^1 ))  \to   \pi_{*}(\Hom(  \tE^0 , \tE^2) ),
\end{equation}
together with the cup product on cochains, defines a map of cohomology groups over $Q$
\begin{equation}
 H^{*}(Q,  \pi_{*}( \Hom(  \tE^1 , \tE^2 ) ) )  \otimes  H^{*}(Q,  \pi_{*}(\Hom(  \tE^0 , \tE^1 ) ))  \to  H^{*}(Q,  \pi_{*}(\Hom(  \tE^0 , \tE^2) ) ). 
 \end{equation}
By Lemma \ref{lem:floer_class}, under the identification between Floer and ordinary cohomology, this product agrees with the one counting holomorphic discs
   \begin{equation}
HW^{*}(Q,  \pi_{*}( \Hom(  \tE^1 , \tE^2 ) ) )  \otimes  HW^{*}(Q,  \pi_{*}(\Hom(  \tE^0 , \tE^1 ) ))   \to  HW^{*}(Q,  \pi_{*}(\Hom(  \tE^0 , \tE^2) ) ).
   \end{equation}
Finally, Lemma \ref{lem:compare_Floer_product_structures} and its analogue in classical cohomology show that the corresponding products defined on the cover
\begin{align}
 H^{*}(\tQ,  \Hom(  \tE^1 , \tE^2 ) )  \otimes  H^{*}(\tQ,  \Hom(  \tE^0 , \tE^1 ) ) & \to H^{*}(\tQ,  \Hom(  \tE^0 , \tE^2) )  \\
 HW^{*}( \tE^1 , \tE^2  )  \otimes  HW^{*}( \tE^0 , \tE^1  ) & \to HW^{*}(  \tE^0 , \tE^2 ) 
\end{align}
also agree, which proves the desired result.
\end{proof}

\appendix
 
\section{Categories of modules over (co)-connective $A_{\infty}$ algebras} \label{sec:DGI}
Given an $A_{\infty}$ algebra $R$, recall that a right $A_{\infty}$ module consists of a graded vector space $P$ together with a collection of maps 
\begin{equation}
\mu^{1|d} \co  P \otimes R^{d}\to P  
\end{equation}
of degree $1-d$.  One of the most useful facts about such algebras is called the homological perturbation lemma (see \cite{kad}):
\begin{lem}
There exists an $A_{\infty}$ structure on $H^*(R)$, together with an $A_{\infty}$ quasi-isomorphism
\begin{equation}
  R \to H^*(R).
\end{equation}
Moreover, if $P$ is an $A_{\infty}$ module, then there also exists a quasi-isomorphic  $A_{\infty}$-module structure on $H^*(P)$. \qed
\end{lem}
This result is usually not stated for $A_{\infty}$ modules, but follows quite easily by considering the $A_{\infty}$ algebra $A \oplus P$.  Algebras for which $\mu^1$ vanishes, and modules for which $\mu^{1|0}$ vanishes are called minimal.  In this section, we shall only consider such algebras and modules.

Let us now consider the case of a minimal $A_{\infty}$ algebra $R$ which is supported in non-positive degrees (the reader should have in mind the wrapped Floer cohomology of a fibre in a cotangent bundle).  A discussion of the sort of results we shall use in the setting of spectra is given in \cite{DGI} where they call such algebras connective.  

Let $P$ be a minimal module over $R$, and denote by $P^{i}$ its component in degree $i$, and $P^{\leq i}$ the direct sum of the components in degree less than $i$.  The operation $\mu^{1|d}$ has strictly negative degree, so its restriction to $R^{d} \otimes P^{\leq i}$ has image in $P^{\leq i} $, which implies the following result:
\begin{lem} \label{lem:filtration_degree}
The ascending degree filtration on a minimal $A_{\infty}$ module is a filtration by submodules.  In particular, a module supported in finitely many cohomological degrees is an iterated extension of modules supported in a single degree. \qed
\end{lem}

This reduces the study of such modules to ones supported in a single degree.  For such a module, the operations $ \mu^{1|d} $ vanish unless $d=1$ for otherwise their degree is negative:
\begin{lem} \label{lem:determined_by_vector_space}
Every module whose cohomology is supported in a single degree is  determined up to $A_{\infty}$ quasi-isomorphism by the corresponding ordinary module over $R^{0}$. In particular, if $R^{0}$ has rank one, then the module structure is determined by the vector space structure.  \qed
\end{lem}

We now consider an $A_{\infty}$ algebra $S$ (still minimal) with endomorphism algebra supported in non-negative degrees.  We consider a twisted complex of (free) $S$-modules of the form 
\begin{equation}
  P =( \oplus_{i= 0}^{D} V_{i}[i] \otimes S, \delta_{i,j})
\end{equation}
such that $V_{i}$ is a vector space, and $\delta_{i,j}$ vanishes whenever $ j < i$.   Because we have assumed that the $i$\th term in the twisted complex is supported in degree $-i$, the first nontrivial differential is a map
\begin{equation}
  V_{0}  \otimes S \to V_{1}[1] \otimes S
\end{equation}
of degree $1$, which is determined by a map
\begin{equation}
  V_{0} \to V_{1} \otimes S^{2},
\end{equation}
where $S^{2}$ consist of degree $2$ elements in $S$.  Here, we use the fact that left multiplication defines a canonical ring isomorphism from $S$ to the endomorphisms of $S$ as a right module over itself, allowing us to identify such endomorphisms with elements $S$. More generally, all morphisms $\delta_{i,j}$ which are non-zero correspond to elements of $S$ of degree greater than or equal to $2$.
\begin{lem} \label{lem:co-connective_result}
 If the endomorphism algebra of the twisted complex $P = ( \oplus_{i= 0}^{D} V_{i} [i]\otimes S, \delta_{i,j}) $ is supported in non-negative degrees, then it is isomorphic, up to shift, to $V \otimes S$ for some vector space $V$.
\end{lem}
\begin{proof}
Let us assume from the start that $V_{0}$ and $V_{D}$ do not vanish, and prove that $D=0$. The endomorphism algebra of a twisted complex is the cohomology of the complex
\begin{equation}
 \Hom^{k}(P,P) =  \bigoplus_{0 \leq i,j \leq D} \Hom_{\bF_{2} }(V_{i} ,V_{j} \otimes S^{k+j-i })
\end{equation}
with differential obtained by taking all possible higher compositions of an element of this complex with the morphisms $\delta_{i,j}$
\begin{equation} \label{eq:differential_twisted_complex_endomorphism}
  x \mapsto \sum_{d=2}^{\infty} \mu^{d}(\delta_{i_d,i_{d-1}}, \ldots, \delta_{i_k,j}, x, \delta_{i_{k-1}, i}, \ldots, \delta_{i_0,i_1}).
\end{equation}
Now, consider a non-zero element 
\begin{equation}
  x \in  \Hom_{\bF_{2} }(V_{0} ,V_{D} \otimes S^{0})  \subset \Hom^{-D}(P,P) 
\end{equation}
which exists by our assumption that both $V_{0}$ and $V_{D}$ are non-zero.  We claim that the corresponding element in $\Hom^{-D}(P,P)$ survives to cohomology.  The fact that it is a cycle follows immediately from knowing that $\delta_{i,0}$ and $\delta_{D,j}$ both necessarily vanish for all $i$ and all $j$ because $0$ is minimal and $D$ maximal in our indexing set.

To see that this element cannot be exact, recall that all morphisms $ \delta_{i,j} $ arise from the subspace of $S$ generated by elements of degree greater than or equal to $2$.  Combining this with the fact that the higher product $\mu^{d}$ has degree $1-d$ implies that each term in the sum appearing in \eqref{eq:differential_twisted_complex_endomorphism} has degree strictly greater than $\deg(x)+1$ (here, we do not take into account any homological shift on $V_{i}$). In particular, since $ S $ is supported in non-negative degrees by assumption, no element of degree $0$ in $S $ can appear in the image of the differential.
\end{proof}

\section{Floer, Morse, and simplicial cohomology of local systems} \label{sec:floer-morse-class}
The set of differential graded local systems on a topological space forms a category with morphism spaces between $E^1$ and $E^2$ given by the cohomology groups  with coefficients in the differential graded local system $\Hom(E^1,E^2)$.  For manifold (or, more generally, spaces admitting a triangulation), one can define a natural differential graded enhancement of this category by considering simplicial cochains with coefficients in such local systems.  The product on cochains is induced by the existence of a natural map of local systems
\begin{equation}
\Hom(E^2,E^3)  \otimes  \Hom(E^1,E^2) \to \Hom(E^1,E^2) 
\end{equation}
given by composition.  This leads us to consider the following slightly more general situation:  Consider a triple of differential graded local systems $E^{0,1}$, $E^{1,2}$ and $E^{0,2}$ on a manifold $Q$, and a map of local systems
\begin{equation} \label{eq:product_of_local_systems}
  E^{1,2} \otimes E^{0,1} \to E^{0,2}.
\end{equation}
Together with the cup product on cochains, this map gives a cohomological multiplication map
\begin{equation} \label{eq:product_cochains}
  H^{*}(Q,E^{1,2}) \otimes  H^{*}(Q,E^{0,1}) \to  H^{*}(Q,E^{0,2}). 
\end{equation}
In Section \ref{sec:gener-prod}, we used the count of holomorphic triangles to define a product
\begin{equation}\label{eq:product_Floer}
  HW^{*}(Q,E^{1,2}) \otimes  HW^{*}(Q,E^{0,1}) \to  HW^{*}(Q,E^{0,2})
\end{equation}
which, whenever $E^{i,j} = \Hom(E^i,E^j)   $ agrees with the cohomological product induced by $\mu^{2}$.  
\begin{lem} \label{lem:floer_class}
 If $E$  is a differential graded local system on a closed exact Lagrangian $Q$, we have an isomorphism
\begin{equation} \label{eq:floer=classical}
  HW^{*}(Q,E) = H^{*}(Q, E),
\end{equation}
which is compatible with the product structure coming from Equations \eqref{eq:product_cochains} and \eqref{eq:product_Floer}.
\end{lem}
\begin{rem}
Extending the results of \cite{plumbings}, one can show that these categories are in fact $A_{\infty}$ quasi-isomorphic.  As we shall not use such a result,  we focus on the cohomological version which is all that we need.
\end{rem}

We start, as in Section \ref{sec:gener-prod}, by choosing a Hamiltonian which is $C^2$-small near $Q$, so that time-$1$ Hamiltonian chords with boundary on $Q$ are in bijective correspondence with the critical points of some Morse function $f \co Q \to \bR$.   In particular, the complex
\begin{equation} 
CW^{*}(Q, E) = \bigoplus_{ \substack{x \in \tQ \\ df(x) = 0}}  E_x
\end{equation}
is also the graded vector space underlying the Morse chain complex of $f$ with coefficients in $E$. In this complex, the differentials are obtained instead by taking the sum, over all negative gradient flow lines of $f$, of the parallel transport maps between the fibres of the local system at critical points whose index differs by one.  The proof that these two complexes are quasi-isomorphic is omitted, and follows either from a degeneration argument going back to Floer \cite{floer} which shows that for a special choice of almost complex structures, there is a bijection between the set of holomorphic strips and of gradient flow lines, or by defining a chain map using mixed moduli space of gradient flow lines and holomorphic discs as in \cite{plumbings}.  

At the level of morphisms, it remains to prove that the cohomology of the Morse complex with local coefficients 
\begin{equation} 
CM^{*}(Q, E) = \bigoplus_{ \substack{x \in Q \\ df(x) = 0}}  E_x
\end{equation}
recovers classical cohomology.    This is a general result about Morse cohomology, which unfortunately does not seem to be in the literature.  We explain one strategy to prove it:  choose a simplicial triangulation $\sQ$ of $Q$ and define
\begin{equation} \label{eq:ordinary_twisted_cochains}
  C^{k}(Q, E) = \bigoplus_{ \substack{\sigma \in \sQ \\ \dim(\sigma) = k}} E_{b_{\sigma}}
\end{equation}
where $\sigma$ is a simplex in $\sQ$ with barycenter $b_{\sigma}$.  Note that whenever we have a codimension $1$ boundary simplex $\tau \subset \partial \sigma$, we have a canonical path from the barycenter of $\tau$ to that of $\sigma$. We define the differential on $  C^{*}(Q,E) $ to be the sum of the parallel transport maps associated to all such paths.  The cohomology of this complex is the classical cohomology of $Q$ with coefficients in $ E$.

In order to pass from the Morse complex to the simplicial complex, we consider the space $\Tree(\sigma,x)$ of gradient flow lines for $f$
\begin{equation}
  [1,+\infty) \to Q
\end{equation}
which at $+\infty$ converge to a critical point $x$, and at $1$ take value on a simplex $\sigma$.  If the simplicial triangulation is chosen generically with respect to $f$, this space is a smooth manifold of dimension $\dim(\sigma) - \deg(x)$, where $\deg(x)$ is the dimension of the ascending manifold of $x$.  In particular, whenever the dimension of $\sigma$ is equal to that of this ascending manifold, there are finitely many such lines, so the sum of all parallel transport maps along these gradient lines (composed with parallel transport from the starting point of the flow line to the barycenter along a path contained in $\sigma$) define a chain map
\begin{equation} \label{eq:morse-to-simplicial}
  CM^{*}(Q,E)  \to C^{*}(Q, E).
\end{equation}

To construct a map in the other direction, we consider a subdivision $\csQ$ of $Q$ into polyhedra which are dual to the cells of $\sQ$.  In particular, for each cell $\sigma \in \sQ$, we have a unique cell $\csigma \in \csQ$ of equal codimension which intersects it, and the intersection between them consists of a single point which may be chosen to be the barycentre $b_{\sigma}$.

We now consider the space $\Tree(x, \csigma)$ of gradient flow lines for $f$
\begin{equation}
  (-\infty,-1 ] \to Q
\end{equation}
which at $-\infty$ converge to a critical point $x$, and at $-1$ take value on a dual cell $\csigma$.  This space is again a smooth manifold of dimension $ \dim(\sigma) - \deg(x) $, so parallel transport along these gradient lines defines a chain map
\begin{equation}\label{eq:simplicial-to-morse}
  C^{*}(Q, E) \to  CM^{*}(Q,E) .
\end{equation}

\begin{lem}
Equations \eqref{eq:morse-to-simplicial} and \eqref{eq:simplicial-to-morse} descend to isomorphisms on cohomology.
\end{lem}
\begin{proof}
We must analyse both compositions, and show they are homotopic to the identity.  The easier composition to study is
\begin{equation} \label{eq:composition_one_direction}
   C^{*}(Q,E) \to  CM^{*}(Q, E) \to C^{*}(Q, E).
\end{equation}
The composition uses the parallel transport maps associated to a gradient flow line starting at $\sigma$ and converging to $x$ at $-\infty$ and another converging to $x$ at $+\infty$ and ending at a cell $\ctau$ of the dual subdivision.  The result of applying the gluing theorem to such flow lines is a gradient flow segment
\begin{equation}
  [-S,S] \to Q
\end{equation}
for some very large constant $S$, which at $S$ lands in $\sigma$, and at $-S$ in $\ctau$.  By considering the moduli space of such flow lines of arbitrary length, we obtain a homotopy between the composition in Equation \eqref{eq:composition_one_direction}, and the map induced by rigid gradient flow lines of length $0$.  Such flow lines correspond to intersection points between cells $\sigma$ and $\ctau$ of complementary dimension, and  the assumption that $\sQ$ and $\csQ$ are dual subdivisions implies that there is exactly one such intersection point, occurring when $\tau = \sigma$.  This proves that the composition in Equation \eqref{eq:composition_one_direction} is homotopic to the identity.

We now consider the composition in the other direction.  Given critical points $x$ and $y$, the associated map $E_x \to E_y$ is obtained by parallel transport along curves parametrised by
\begin{equation}
   \cup_{\sigma \in \sQ} \Tree(y, \csigma) \times \Tree(\sigma,x).
\end{equation}
Note that we can interpret this as the space of gradient flow lines of the function $f(q_1) - f(q_2)$ on $Q \times Q$, which start at $(x,y)$ and end on the chain
\begin{equation}
  \cup \sigma \times \csigma.
\end{equation}
Since $Q$ is a manifold, this chain in fact a cycle, which is homologous to the diagonal.  The homology between the diagonal and this simplicial approximation defines a homotopy between the compositions of Equations \eqref{eq:morse-to-simplicial} and \eqref{eq:simplicial-to-morse}, and the map induced by flow lines starting at $(x,y) $ and ending on the diagonal.  The only such flow lines which are rigid are constant, so this map is the identity.
\end{proof}

Next, we prove that the product structures in Floer and simplicial cohomology agree.  Again, we use the product on the Morse complex as an intermediary.  The comparison between the product on Floer complexes and Morse complexes can be done using Fukaya and Oh's degeneration argument \cite{FO}, which shows that for a specifically chosen almost complex structure, the counts of holomorphic triangles and of Morse trees agree (see also \cite{ekholm,vito}).  Alternatively, the methods of \cite{plumbings} give a degeneration-free construction of a homotopy between the two products.

In order to define the product on the Morse complex, we consider maps $ \gamma \co T \to Q$ from a trivalent metric tree with two incoming infinite edges which converge to critical points $x_1$ and $x_2$ of $f$, and an infinite outgoing edge converging to a critical point $x_0$.  In particular, the incoming edges are isometric to $[0,+\infty)$ and the outgoing edge to $(-\infty,0]$; we write $t_e$ for this fixed parametrisation of an edge $e$.  We choose a family of vector fields $Y$ on $Q$, parametrised by points on the edges of $T$, which agrees with the gradient vector field of $f$ away from a compact set in $T$.  We define $\Tree(x_0,x_1,x_2)$ to be the space of maps satisfying these asymptotic conditions, as well as the differential equation
\begin{equation}
 \frac{ d \gamma}{d t_{e}} = Y_{t_e}
\end{equation}
along each edge.

Let us now return to the setting of Equation \eqref{eq:product_of_local_systems}, in which we have a map from the tensor product of differential graded local systems $E^{0,1}$ and $E^{1,2}$ to a local system $E^{0,2}$.  By parallel transport along the edges of a tree $\gamma$, we obtain maps
\begin{align}
  \gamma^{0,1} \co E^{0,1}_{x_1} & \to E^{0,1}_{x_0}  \\
\gamma^{1,2} \co E^{1,2}_{x_2} & \to E^{1,2}_{x_0} .
\end{align}
For generic choices of the vector field $Y  $, the space $ \Tree(x_0,x_1,x_2) $ is a manifold of dimension $\deg(x_0) - \deg(x_1) - \deg(x_2)$, which is the interior of a compact manifold.  With this in mind, we define the product on Morse complexes to be
\begin{align}
CM^{*}(Q,   E^{1,2}) \otimes CM^{*}(Q,   E^{0,1}) & \to CM^{*}(Q,   E^{0,2}) \\
v_{1,2} \otimes v_{0,1} & \mapsto \sum_{\gamma} \gamma^{1,2}(v_{1,2})  \cdot \gamma^{0,1}(v_{0,1})
\end{align}
where the sum is taken over all trees lying in moduli spaces of dimension $0$.

In order to check that at the cohomological level, the Morse-theoretic product agrees with the one coming from simplicial cochains, we recall that in the absence of local systems, the product of generators dual to cells $\sigma$ and $\tau$ either vanishes, or is equal to the dual of a cell $\rho$.  In terms of the dual subdivision, one interprets this cup product as a (perturbed) intersection product between $\csigma  $ and $\ctau$, along a (small perturbation) of the dual cell $\crho$ (see \cite[Lemma 2.3]{plumbings}).   Choose arbitrary paths 
\begin{align}
  \gamma_{\sigma, \rho} & \co [0,1] \to \rho \\
\gamma_{\tau, \rho} & \co [0,1] \to \rho
\end{align}
connecting the barycenters of $\sigma$ and $\tau$ to that of $\rho$, and write $  \gamma_{\sigma, \rho}^{1,2} $  and   $  \gamma_{\tau, \rho}^{0,1} $ for the associated parallel transport maps of the local systems $E^{1,2}$ and $E^{0,1}$.   

\begin{defin}
The product on simplicial cochains
\begin{equation}
C^{*}(Q,   E^{1,2}) \otimes C^{*}(Q,   E^{0,1})  \to C^{*}(Q,   E^{0,2})
\end{equation}
is given, on $ E^{1,2}_{b_{\sigma}} \otimes E^{0,1}_{b_{\tau}}$ by the formula
\begin{equation}
v_{b_{\sigma}} \otimes v_{b_{\tau}} \mapsto   \gamma_{\sigma, \rho}^{1,2}  (v_{b_{\sigma}})   \cdot \gamma_{\sigma, \rho}^{0,1}(v_{b_{\tau}})
\end{equation}
\end{defin}

In \cite[Section 3]{plumbings}, a proof is given, in the absence of local systems, that the map from simplicial to Morse cochains defined in Equation \eqref{eq:simplicial-to-morse} is compatible with products at the cohomological level.  The key idea is to consider maps from a trivalent tree $T_{a}$ with one infinite (outgoing) edge converging to a critical point, and two incoming edges of length $a$ whose endpoints lie on cells $\csigma  $ and $\ctau$ of the dual subdivision.  When $a=0$, the incoming points agree, and one can choose an appropriate deformation of the gradient flow equation so that they indeed lie on the cell corresponding to $  \csigma \cup \ctau$.  As $a$ converges to $+\infty$, such trees converge to gradient flow lines from $\csigma$ and $\ctau$ to critical points of $f$, together with a trivalent tree all of whose edges are infinite.

Having laid out the definition of all the relevant product structures in the presence of local systems, we leave the curious reader to check that the moduli space of trivalent trees with varying parameter $a$ defines a homotopy for the diagram
\begin{equation}
  \xymatrix{  C^{*}(Q,   E^{1,2}) \otimes C^{*}(Q,   E^{0,1})  \ar[r] \ar[d] &  C^{*}(Q,   E^{0,2}) \ar[d] \\
CM^{*}(Q,   E^{1,2}) \otimes CM^{*}(Q,   E^{0,1}) \ar[r] & CM^{*}(Q,   E^{0,2}). }
\end{equation}

\section{Equivalence with the zero section} \label{sec:FSS}
In Section 5.1 \cite{FSS2}, Fukaya, Seidel and Smith discuss a finite covering trick proving the following result (see also Corollary 1.3 of their paper):
\begin{lem} \label{lem:FSS}
  Every closed exact Lagrangian of vanishing Maslov class, which is embedded  in the cotangent bundle of a closed oriented manifold is equivalent to the zero section equipped with some local system of rank one.   Moreover, the inclusion induces an isomorphism on cohomology.
\end{lem}
As the foundations of the wrapped Fukaya category were not developed at the time, the proof was not complete because the other two models for the Fukaya category of a cotangent bundle were not technically adequate to prove the result:  (1) the Lefschetz pencil approach to studying Fukaya categories, in addition to yielding results only away from characteristic $2$, requires making a choice of embedding into affine varieties which vary as we pass to different covers and  (2) the constructible sheaf approach used by Nadler in \cite{nadler}, while in principle applicable over any field, is only defined in the literature over the reals.

To complete this paper, we shall therefore give a wrapped Fukaya category version of their argument.   Let $N$ be a closed manifold, and $Q \in \TN$ a Lagrangian satisfying the hypothesis of Lemma \ref{lem:FSS}.  Let $b \in H^{2}(\TN, \bF_{2})$ denote the pullback of the second Stiefel-Whitney class of $N$, and $\Wrap_{b}(\TN)$ the wrapped Fukaya category of $\TN$, over the integers, twisted by this class (see \cite{fibre-generate} for the definition).  The objects of this category are exact Lagrangians which are graded and relatively spin, i.e. for which the restriction of $b$ agrees with the second Stiefel-Whitney class.  For the moment we shall assume that $Q$ is indeed relatively spin so that Floer cohomology makes sense over the integers. We start with the main  result of  \cite{fibre-generate}, which asserts that the twisted wrapped Fukaya category $\Wrap_{b}(\TN)$  is generated by a fibre. Since $HW^{*}_{b}(\Tn)$ is supported in non-negative degrees, $\bZ$-graded modules whose cohomology groups are supported in a single cohomological degree are determined up to $A_{\infty}$ quasi-isomorphism by the corresponding cohomological module.  In particular, it shall suffice to prove that $ HW^{*}_{b}(\Tn, Q)$, over the integers, is a free abelian group of rank $1$,  as such an object is represented by the appropriate local system on $N$ in the category of modules over $ CW^{*}_{b}(\Tn) $.

The filtration of $ HW^{*}_{b}(\Tn, Q)  $ by degrees shows that $Q$ is an iterated extension of modules over $  HW^{*}_{b}(\Tn)$ each supported in a single cohomological degree, i.e. an iterated extension of the zero section equipped with some representation of its fundamental group in $GL(k,\bZ)$ for some integers $k$.  For each prime $p$, we pick a cover $\tN_{p}$ in which all these representations are trivial modulo $p$.  The pullback functor $\Wrap_{b}(\TN) \to   \Wrap_{b}(\TtN_{p})$ implies that the inverse image $\tQ_{p}$ of $Q$ must be quasi-isomorphic to an iterated extension of the zero section in $\TtN_{p}$.  Lemma \ref{lem:co-connective_result} (or a spectral sequence argument as in the introduction of \cite{FSS2}) shows that this is only possible if  $\tQ_{p}$ is in fact quasi-isomorphic to the zero section.  We conclude that $ HW^{*}_{b}(\Ttn_{p}, \tQ_{p}) $ with $\bZ_{p}$ coefficients has rank $1$, which implies that $ HW^{*}_{b}(\Tn, Q)  $ also has rank $1$ over $ \bZ_{p}$ for every prime $p$.  But as  $ HW^{*}_{p}(\Tn, Q)  $ is a priori a finitely generated abelian group, we conclude that it must be free of rank $1$.

To drop the assumption that $Q$ is relatively spin, consider the following  argument due to Paul Seidel: the preceding discussion,  applied over $\bF_{2}$ implies that the map $Q \to N$ induces an isomorphism on cohomology with $\bF_{2}$ coefficients.  In particular, the Steenrod square operations on $H^{*}(Q) $ and $H^{*}(N) $ agree.  But the Wu formula (see, e.g. \cite[Lemma 11.13]{milnor-stasheff}) expresses the Stiefel-Whitney classes in terms of Steenrod squares, so we conclude that $w_{2}(N)$ pulls back to $w_{2}(Q)$, i.e. that $Q$ is relatively spin.

\end{document}